\documentclass{article}
\usepackage[utf8]{inputenc}
\usepackage{geometry}[margin=1in]
\usepackage{amsmath}
\usepackage{nicematrix}
\usepackage{amssymb}
\usepackage{amsthm}
\usepackage{enumitem}
\usepackage{pgfplots}
\usepackage{svg}
\usepackage{graphicx}
\usepackage{comment}
\usepackage{tikz}
\usepackage{tikz-cd}
\usepackage{indentfirst}
\usepackage{titlesec}
\usepackage{stmaryrd}
\usepackage[hidelinks]{hyperref}
\usepackage[maxbibnames=99,
backend=biber,
style=alphabetic,
]{biblatex}
\hypersetup{colorlinks=true, linkcolor = blue, citecolor = blue}
\usepackage{graphicx,import}

\newcommand{\zz}{\ensuremath{\mathbb{Z}}}
\newcommand{\qq}{\ensuremath{\mathbb{Q}}}

\newcommand{\del}{\partial}

\newcommand{\la}{\langle}
\newcommand{\ra}{\rangle}

\renewcommand{\cal}[1]{\mathcal{#1}}
\renewcommand{\tilde}[1]{\widetilde{#1}}
\renewcommand{\hat}[1]{\widehat{#1}}

\newcommand{\pmd}{\ensuremath{{\pm\circ}}}

\newcommand{\norm}[1]{\left\lVert#1\right\rVert}

\DeclareMathOperator{\Mod}{Mod}
\DeclareMathOperator{\SMod}{SMod}
\DeclareMathOperator{\Lk}{Lk}
\DeclareMathOperator{\Ind}{Ind}

\theoremstyle{definition}
\newtheorem{thm}{Theorem}[section]
\newtheorem{defn}[thm]{Definition}
\newtheorem{lem}[thm]{Lemma}
\newtheorem{cor}[thm]{Corollary}
\newtheorem{prop}[thm]{Proposition}
\newtheorem{rem}[thm]{Remark}
\newtheorem{que}[thm]{Question}
\newtheorem{ex}[thm]{Example}
\newtheorem*{ack}{Acknowledgements}
\newtheorem{mainthm}{Theorem}

\newtheorem{subthm}{Theorem}

\titleformat{\section}{\Large\bfseries\filcenter}{\thesection.}{1em}{}

\title{Fixed-point-free pseudo-Anosov homeomorphisms, knot Floer homology and the cinquefoil}
\author{Ethan Farber, Braeden Reinoso and Luya Wang}
\date{}

\begin{document}

\maketitle

\begin{abstract}
    Given any genus-two, hyperbolic, fibered knot in $S^3$ with nonzero fractional Dehn twist coefficient, we show that its pseudo-Anosov representative has a fixed point. Combined with recent work of Baldwin--Hu--Sivek, this proves that knot Floer homology detects the cinquefoil knot $T(2,5)$, and that the cinquefoil is the only genus-two L-space knot in $S^3$. Our results have applications to Floer homology of cyclic branched covers over knots in $S^3$, to $\mathit{SU}(2)$-abelian Dehn surgeries, and to Khovanov and annular Khovanov homology. Along the way to proving our fixed point result, we describe a small list of train tracks carrying all pseudo-Anosov homeomorphisms in most strata on the punctured disk. As a consequence, we find a canonical track $\tau$ carrying all pseudo-Anosov homeomorphisms in a particular stratum $\mathcal{Q}_0$ on the genus-two surface, and describe every fixed-point-free pseudo-Anosov homeomorphism in $\mathcal{Q}_0$.
\end{abstract}

\section{Introduction}
Recent developments in Heegaard Floer homology have highlighted intimate connections between link homology theories and the dynamics of surface diffeomorphisms (see e.g. \cite{BHS}, \cite{Ni1}, \cite{Ni2}, \cite{GS}). The aim of this paper is to use tools which may be familiar to some dynamicists, in order to study a particular open question in Heegaard Floer homology: whether knot Floer homology detects the torus knot $T(2,5)$. To that end, we answer this question in the affirmative:

\begin{mainthm}\label{t25thm}
If $\widehat{\mathit{HFK}}(K;\qq)\cong\widehat{\mathit{HFK}}(T(2,5);\qq)$ as bi-graded vector spaces, then $K=T(2,5).$ In particular, $T(2,5)$ is the only genus-two L-space knot in $S^3$.
\end{mainthm}

This is the first knot Floer detection result for a knot of genus two. Prior detection results rely on a classification of fibered knots with genus at most one \cite{Oz_Sz_unknot, Ghiggini}. But, there are infinitely many genus two, hyperbolic, fibered knots with the same Alexander polynomial as that of $T(2,5)$ \cite{Misev}, indicating that close attention must be paid to the structure of the fibration. In this vein, our proof completes a strategy outlined by Baldwin--Hu--Sivek in \cite{BHS}, which uses connections between knot Floer homology and symplectic Floer homology, to reduce the question to a problem about fixed points of pseudo-Anosov maps. Our first key result is a solution to this problem.

\begin{defn}
Let $K\subset Y$ be a hyperbolic, fibered knot in a 3-manifold $Y$. The knot $K$ specifies an open book decomposition $(S,h)$ for $Y$, where $h:S\to S$ is freely isotopic to a pseudo-Anosov $\psi$ on $S$. We say that $K$ is \emph{fixed point free} if $\psi$ has no fixed points in the interior of $S$.
\end{defn}

\begin{mainthm}\label{fpfthm}
Let $K$ be a hyperbolic, genus-two, fibered knot in $S^3$. If the fractional Dehn twist coefficient $c(K)\neq 0$, then $K$ is not fixed point free.
\end{mainthm}

Pseudo-Anosov maps and fractional Dehn twist coefficients are defined in Section \ref{background}. Theorem \ref{t25thm} follows immediately from Theorem \ref{fpfthm} and the work of Baldwin--Hu--Sivek (\cite{BHS}, Theorem 3.5). We would like to emphasize that the proof of Theorem \ref{fpfthm} is completely geometric in nature, and after the introduction we will only make passing references to Floer homology theories throughout the paper. An unfamiliar reader need not have expertise in any link homology theory to understand the proof of Theorem \ref{fpfthm}.

The proof of Theorem \ref{fpfthm} is broken down into two smaller theorems (Theorems \ref{6case} and \ref{221thm}), based on cases for the singularity type of $\psi$. An outline for the proof is given in subsection \ref{outline}. We will now discuss various applications of our techniques and results.

\subsection{Applications to train tracks and the dilatation spectrum}

One of the central tools in the proof of Theorem \ref{fpfthm} is the theory of train tracks for pseudo-Anosov braids, including a theory of ``tight splitting" developed in Section \ref{sec:split}. We believe the techniques we use are broadly applicable elsewhere within surface dynamics. For example:

\begin{mainthm}[cf. Theorem \ref{thm:221}]\label{carrythm}
Let $\psi$ be a pseudo-Anosov on the genus-two surface with one boundary component, with singularity type $(4;\emptyset;3^2)$. Then, $\psi$ is conjugate to a pseudo-Anosov carried by the train track depicted on the bottom left in Figure \ref{liftedtt}. A similar statement holds for the closed genus-two surface.
\end{mainthm}

This result suggests a strategy for systematically studying the set of dilatations of pseudo-Anosovs in genus two. Indeed, Theorem \ref{carrythm} reduces the study of dilatations in the stratum $(4;\emptyset;3^2)$ to the study of a special collection of maps on a single graph. One would hope for the development of a \textit{kneading theory} generalizing the classical theory of Milnor-Thurston for interval maps (cf. \cite{MT}). Applying such a theory to a small list of tracks in each of the other strata in genus two would provide a much clearer understanding of the full dilatation spectrum.  This line of study was suggested to us by Chenxi Wu. Another step in this direction is the following result:

\begin{mainthm}\label{carrythm2}
Let $\psi$ be a pseudo-Anosov on the punctured disk with at least one $k$-pronged singularity away from the boundary with $k\geq 2$. Then $\psi$ is carried by a standardly embedded train track $\tau$ with no joints.
\end{mainthm}

See Definition \ref{defn:stdembed} for the definition of a standardly embedded train track. A \textit{loop switch} of a standardly embedded train track $\tau$ is a switch at a loop surrounding a 1-prong singularity (cf. Definition \ref{defn:loopswitch}), and a \textit{joint} is a loop switch that is incident to more than one expanding edge (cf. Definition \ref{defn:joint}). The track in Theorem \ref{carrythm} is the lift of a joint-less track on the punctured disk, so Theorem \ref{carrythm} may be seen as a specific case of Theorem \ref{carrythm2}. See subsection \ref{splittingsection} for the proofs of Theorems \ref{carrythm} and \ref{carrythm2}.

\subsection{Applications to the Floer homology of branched covers}

For a knot $K\subset S^3$, let $\Sigma_n(K)$ denote the $n$-fold cyclic cover of $S^3$ branched along $K$.
There has been much interest recently in the Floer homology of $\Sigma_n(K)$ in terms of $K$. For example, Boileau--Boyer--Gordon have studied extensively in \cite{BBG} and \cite{BBG2} the set of all integers $n \geq 2$ such that $\Sigma_n(K)$ is an L-space (see also e.g. \cite{IT} and \cite{Pet}). One question that has persisted in this area is the following:

\begin{que}[Boileau--Boyer--Gordon, Moore]
Can $\Sigma_n(K)$ be an L-space for $K$ a hyperbolic L-space knot?
\end{que}

\noindent Combining Theorem \ref{t25thm} with (\cite{BBG}, Corollary 1.4) yields the following complete answer to this question for $n>2$:

\begin{cor}\label{branch}
If $K$ is an L-space knot and $\Sigma_n(K)$ is an L-space for some $n>2$, then $K$ is either $T(2,3)$ or $T(2,5)$. In particular, $K$ is not hyperbolic.
\end{cor}

\subsection{Applications to instanton Floer homology and Dehn surgery}
For a 3-manifold $Y$, let $R(Y)=\text{Hom}(\pi_1(Y),\mathit{SU}(2))$ denote the $\mathit{SU}(2)$-representation variety. We say that a 3-manifold $Y$ is \emph{$\mathit{SU}(2)$-abelian} if $R(Y)$ contains no irreducibles. The name is motivated by the fact that $Y$ is $\mathit{SU}(2)$-abelian if and only if every $\rho\in R(Y)$ has abelian image.

Following work initiated by Kronheimer--Mrowka in their proof of the Property P conjecture \cite{KM1}, Baldwin--Li--Sivek--Ye \cite{BLSY}, Baldwin--Sivek \cite{BS1}, and Kronheimer--Mrowka \cite{KM2} proved that $r$-surgery $S^3_r(K)$ on a nontrivial knot $K\subset S^3$ is not $\mathit{SU}(2)$-abelian for all slopes $r\in[0,3]\cup [4,5)$ with prime power numerator, and for some additional slopes $r\in[3,4)$.

The key theory which facilitates most of these results is the instanton Floer homology of the surgered manifold $S^3_r(K)$ (and related techniques arising from this theory, as in \cite{KM2}). Combining Theorem \ref{fpfthm} with (\cite{BLSY}, Proposition 2.4) allows us to prove an analogue of Theorem \ref{t25thm} for instanton Floer homology:

\begin{cor}\label{instantonLspace}
The cinquefoil $T(2,5)$ is the only genus-two instanton L-space knot, i.e. the only genus-two knot $K$ for which $\text{dim}~I^\#(S^3_r(K))=|H_1(S^3_r(K))|$ for some $r>0$.
\end{cor}

Now, as described in (\cite{BLSY}, Section 1.3), Corollary \ref{instantonLspace} completes the set of slopes $r$ for which $S^3_r(K)$ is not $\mathit{SU}(2)$-abelian, to all rational numbers $r\in[0,5)$ with prime power numerator:

\begin{cor}
Let $K\subset S^3$ be a nontrivial knot, and $r\in[0,5)$ a rational number with prime power numerator. Then, $S^3_r(K)$ is not $\mathit{SU}(2)$-abelian.
\end{cor}

\begin{rem}
Note that $S^3_r(K)$ may in general be $\mathit{SU}(2)$-abelian for $r\geq 5$, as $S^3_5(T(2,3))$ is the lens space $L(5,1)$, which has abelian fundamental group. However, Baldwin--Li--Sivek--Ye in \cite{BLSY} have extended the slopes for which $S^3_r(K)$ is not $\mathit{SU}(2)$-abelian to some additional $r\in(5,7)$. It is an open question whether $S^3_r(K)$ is $\mathit{SU}(2)$-abelian for \emph{all} rational numbers $r\in[0,5)$, though it is known to be true for $r\in[0,2]$ by work of Kronheimer--Mrowka in \cite{KM2}.
\end{rem}

\subsection{Applications to Khovanov homology}
In \cite{BHS}, Baldwin--Hu--Sivek proved that Khovanov homology (with coefficients in $\zz/2\zz$) detects the cinquefoil $T(2,5)$. Combining Theorem \ref{t25thm} with previous work of Baldwin--Dowlin--Levine--Lidman--Sazdanovic (\cite{BDLLS}, Corollary 2), we can improve Baldwin--Hu--Sivek's result from $\zz/2\zz$-coefficients to $\qq$-coefficients:
\begin{cor}
If $\mathit{Kh}(K;\qq)\cong \mathit{Kh}(T(2,5);\qq)$ as bi-graded $\qq$-vector spaces, then $K=T(2,5)$.
\end{cor}

We also obtain detection results in annular Khovanov homology. One may think of $T(2,5)$ as the lift of the braid axis for the 5-braid $B=\sigma_1\sigma_2\sigma_3\sigma_4$ in $S^3$ seen as the double-branched cover over $\widehat{B}$ under the Birman--Hilden correspondence (see subsection \ref{BHsection} for background). From this perspective, we can adapt techniques of Binns--Martin in (\cite{BM}, Theorems 10.2, 10.4, 10.7) to prove that annular Khovanov homology detects the aforementioned braid closure:
\begin{cor}\label{akhcor}
Let $L\subset A\times I$ be an annular link with $\mathit{AKh}(L;\qq)\cong \mathit{AKh}(\widehat{B};\qq)$. Then, $L$ is isotopic to $\widehat{B}$ in $A\times I$.
\end{cor}

The proof of this corollary is almost identical to those of the analogous results proved by Binns--Martin. So, we omit the proof in the present work, and instead refer the reader to \cite{BM}.

\subsection{Outline of the proof of Theorem \ref{fpfthm}}\label{outline}
Let $K$ be a hyperbolic, genus-two, fibered knot in $S^3$ with associated open book decomposition $(S,h)$ and pseudo-Anosov representative $\psi_h$. If $c(K)\neq 0$ and $K$ is fixed point free, it follows from work of Baldwin--Hu--Sivek (see Theorem \ref{bhsthm}) that:

\begin{enumerate}[label=$\bullet$]
    \item $\psi_h$ has singularity type either:
    \begin{enumerate}[label=\hspace{1cm}Case \arabic*:, left=0pt, itemsep=-1pt, topsep=-5pt]
        \item $(6;\emptyset;\emptyset)$, or
        \item $(4;\emptyset;3^2)$
    \end{enumerate}
    \item $h$ is the lift of a 5-braid $\beta$ with unknotted braid closure $\widehat{\beta}$
    \item The pseudo-Anosov representative $\psi_\beta$ of the braid $\beta$ has singularity type either:
        \begin{enumerate}[label=\hspace{1cm}Case \arabic*:, left=0pt, itemsep=-1pt, topsep=-5pt]
        \item $(3;1^5;\emptyset)$
        \item $(2;1^5;3)$
    \end{enumerate}
\end{enumerate}
For our conventions of the singularity types, see Section \ref{MCGsection}. Cases 1 and 2 are mutually exclusive, and we will use notation $K,h,\beta,\psi_h,\psi_\beta$ as above for the rest of this outline, in either case.

Case 1 is dealt with in Section \ref{6section}. In this case, Masur--Smillie proved in \cite{MS} that the foliations preserved by $\psi_h$ are orientable, so that the dilatation of $\psi_h$ is a root of the Alexander polynomial $\Delta_K(t)$. This is a special fact about the stratum $(6;\emptyset;\emptyset)$ in genus two. Using the Lefschetz fixed point theorem and basic facts about Alexander polynomials of fibered knots in $S^3$, we completely determine the Alexander polynomial of $K$ and conclude that the dilatation of $\psi_h$ coincides with the minimal dilatation $\lambda_2$ for genus-two pseudo-Anosovs. Work of Lanneau--Thifeault in \cite{LT} further implies that $\psi_h$ is the almost unique genus-two pseudo-Anosov realizing $\lambda_2$ as its dilatation. It follows that $\beta$ is (up to inverse and composing with the hyperelliptic involution) conjugate to the dilatation-minimizing 5-braid $\alpha$ from \cite{HS} within the mapping class group of the punctured sphere. We then show that no braid which is conjugate to $\alpha$ within the spherical mapping class group has unknotted closure.

Case 2 is harder: we perform our analysis using a splitting argument and a careful combinatorial analysis of train track maps. In this case, we focus on the braid $\beta$ and its pseudo-Anosov representative $\psi_\beta$. We show (Theorem \ref{thm:221}) that any pseudo-Anosov $\psi_\beta$ on the five-punctured disk in the stratum $(2;1^5;3)$ is carried by a single canonical train track $\tau$ (cf. also Theorems \ref{carrythm} and \ref{carrythm2}). To prove this result, we develop a theory of ``tight splitting" in Section \ref{sec:split}, which allows us to study the splitting of \emph{all} train track maps on any given track within the stratum.

In subsection \ref{unknotsubsection}, we find a collection of braids $\beta_n$ inducing special train track maps $f_n:\tau\to\tau$ on the distinguished track $\tau$ from the previous paragraph. We show, in subsection \ref{traintracksubsection}, that $f_n$ are the \emph{only} maps on $\tau$ which could lift to train track maps for fixed-point-free pseudo-Anosovs $\psi_h$ in the cover. So, if $\beta$ is any braid which lifts to a map $h$ with fixed-point-free pseudo-Anosov representative $\psi_h$, then $\beta$ is conjugate (in the mapping class group of the punctured sphere) to one of $\beta_n$. A similar argument as in Case 1 shows that no braid conjugate to $\beta_n$ has unknotted closure.

\begin{ack}
A special thank you goes to John Baldwin for suggesting this problem to the second author, and for his continued support, encouragement, and generous insights throughout the preparation of this work. We also thank Jacob Caudell, Karl Winsor, Gage Martin, Fraser Binns, J\'er\^ome Los, Steven Sivek, and Chi Cheuk Tsang for very helpful discussions, suggestions, questions, and/or comments. The third author would also like to thank Ian Agol and Paolo Ghiggini for independently suggesting the question to her, and for initial discussions regarding the project. The third author acknowledges support by the NSF GRFP under Grant DGE 2146752.
\end{ack}


\section{Background}\label{background}

\subsection{Mapping classes and fractional Dehn twists}\label{MCGsection}

Let $S=S_{g,n}^r$ be a compact surface of genus $g$ with $n$ marked points and $r$ boundary components. The \textit{mapping class group} of $S$ is the group $\Mod(S)$ of isotopy classes of homeomorphisms $h: S \to S$ which fix $\del S$ point-wise, and permute the marked points of $S$, where the isotopies fix all boundary components and marked points. The \emph{symmetric mapping class group} of $S$ is the analogous group $\SMod(S)$ obtained by additionally requiring that the homeomorphisms commute with the hyperelliptic involution $\iota: S\to S$, i.e. $h\circ\iota=\iota\circ h$.

\begin{figure}
    \centering
    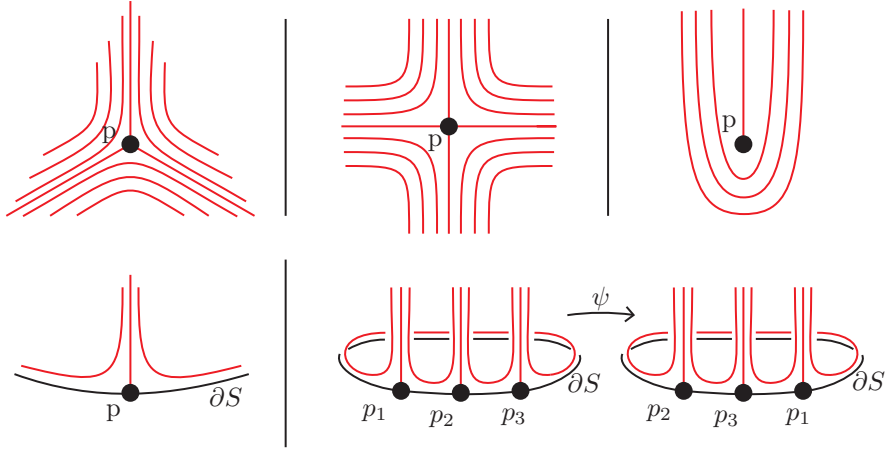
    \caption{Top (left to right): a 3-pronged saddle, a 4-pronged saddle, and a 1-pronged saddle at a marked point. Bottom (left to right): the neighborhood of a boundary singularity, and a 3-pronged boundary which permutes the first prong to the second.}
    \label{singptsfig}
\end{figure}

\begin{defn}
A \textit{pseudo-Anosov} is a homeomorphism $\psi: S \to S$ preserving a pair of transverse singular measured foliations $(\cal{F}^u, \mu^u)$ and $(\cal{F}^s, \mu^s)$ such that
\[
\psi \cdot (\cal{F}^u, \mu^u) = (\cal{F}^u, \lambda \mu^u), \ \text{and} \ \psi \cdot (\cal{F}^s, \mu^s) = (\cal{F}^s, \lambda^{-1} \mu^s)
\]

\noindent for some fixed real number $\lambda>1$, called the \textit{dilatation} of $\psi$.

We further require that each singularity $p$ of $\cal{F}^u$ or $\cal{F}^s$ is a ``$k$-pronged saddle," as in Figure \ref{singptsfig}, where $k\geq 3$ for $p$ in the interior of $S$, or $k\geq 1$ for $p$ a marked point or puncture. Along $\del S$, the singular points must all have a neighborhood of the form shown on the bottom left of Figure \ref{singptsfig}.
\end{defn}

By a \emph{$k$-prong boundary singularity} or a \emph{$k$-prong singularity on the boundary}, we mean that $\psi$ has $k$ singular points on a particular boundary component of $S$. The \emph{singularity type} of $\psi$ is the tuple $(b_1,...,b_r;m_1,...,m_n;k_1,...,k_s)$ where the $i^\text{th}$ boundary component has $b_i$-prongs, the $i^\text{th}$ puncture or marked point has $m_i$-prongs, and the $i^\text{th}$ interior singularity has $k_i$ prongs. We will use $\emptyset$ if $\psi$ has no boundary (or if $\psi$ has no marked points or punctures, or interior singularities), and we will use an exponent to denote a repeated number of prongs. For example, the tuple $(3;1^5;\emptyset)$ indicates that $\psi$ has a 3-pronged boundary, five 1-pronged marked points or punctures, and no interior singularities. A \emph{stratum} on $S$ is the collection of all pseudo-Anosovs on $S$ with a given singularity type.

\begin{thm}[Nielsen--Thurston classification]
Any element $h\in\Mod(S)$ is freely-isotopic rel. punctures (i.e. isotopic through homeomorphisms which fix the punctures, but may rotate a boundary component) to a representative $\psi$ with at least one of the following properties:

\begin{enumerate}[label=(\arabic*)]
    \item $\psi^n=\text{id}$ for some $n$,
    \item $\psi$ preserves the isotopy class of a multicurve $C$ on $S$, or
    \item $\psi$ is pseudo-Anosov. This case is disjoint from the previous two.
\end{enumerate}
\end{thm}

We say that such a $\psi$ satisfying one of the properties (1)---(3) is \emph{geometric}, and in case (3), we refer to $\psi$ as the \emph{pseudo-Anosov representative} of $h$. The representative $\psi$ is unique for any such $h$, although when $S$ has non-empty boundary, $\psi$ will never be isotopic rel. boundary to an element of $\Mod(S)$, as $\psi$ may rotate $\del S$.

The \emph{fractional Dehn twist coefficient} $c(h)=n+m/k$ is a rational number which measures the rotation of $\psi$ along $\del S$. Here, $n$ is an integer measuring the number of full-rotations of $h$ along $\del S$; $k$ is the number of prongs of $\psi$ on $\del S$; and $\psi$ cyclically permutes the endpoint of the first prong to that of the $(m+1)^{th}$ prong. For example, in the bottom right of Figure \ref{singptsfig}, the ``fractional part" of $c(h)$ is $1/3$, where $h\in\Mod(S)$ has pseudo-Anosov representative $\psi$. In particular $c(h)\in\zz$ if and only if the rotation number of $\psi$ along the boundary is zero, and when $\psi$ has a single boundary prong this is always the case.

The fractional Dehn twist coefficient can be extended analogously to braids, thought of as elements of $\Mod(S_{0,n}^1)$ for some $n>1$. In this case, we denote it by $c(\beta)$, where $\beta$ is a braid. When $K$ is a fibered knot with open book decomposition $(S,h)$, we define the fractional Dehn twist coefficient $c(K)$ to be that of $h$, i.e. $c(K):=c(h)$. It is crucial to note that, in general, $c(K)$ is \emph{not} the same as $c(\beta)$ for $\beta$ a braid representative of $K$. The following theorem details a few well-known properties which we will make use of in this paper:

\begin{thm}[\cite{HKM},\cite{Plam},\cite{IK}]\label{fdtcthm}
Let $h:S\to S$ be a mapping class with $\del S$ connected, and let $\beta:S_{0,n}^1\to S_{0,n}^1$ be a braid.
\begin{itemize}
    \item $c(h)$ and $c(\beta)$ are preserved under conjugation.
    \item $c(D_{\del S}^m\circ h^k)=m+kc(h)$ for any $k,m\in\zz$, where $D_{\del S}$ is a Dehn twist along $\del S$.
    \item $c(\Delta^{2m}\beta^k)=m+kc(\beta)$ for any $k,m\in\zz$, where $\Delta^2=(\sigma_1...\sigma_{n-1})^n$.
    \item If $\beta$ is $\sigma_1$-positive (i.e. $\beta$ can be written with only positive powers of $\sigma_1$) then $c(\beta)\geq 0$.
    \item If $\beta$ is a positive pseudo-Anosov braid, then $c(\beta)>0$.
\end{itemize}
\end{thm}

See \cite{Thurston} for more details regarding the Nielsen--Thurston classification, and \cite{HKM} or \cite{KR} for more details regarding fractional Dehn twist coefficients.

\subsection{Knots, braids, and the Birman--Hilden correspondence}\label{BHsection}

We may use branched coverings to understand relationships between mapping class groups of different surfaces. The Birman--Hilden correspondence is a key tool for this study. For our purposes, the correspondence will help us study the mapping class groups of $S_2^1$ and $S_2$ seen as two-fold branched covers over the disk $S_{0,5}^1$ and sphere $S_{0,6}$, respectively, via the hyperelliptic involution $\iota$. Specifically, there is a diagram:

\begin{center}
\begin{tikzcd}
\SMod(S_2^1) \arrow[r, "\text{cap-off}"] \arrow[d, "\Theta_2^1"'] & \SMod(S_{2,1}) \arrow[r, "\text{forget}"] & \Mod(S_2) \arrow[d, "\Theta_2"]\\
\Mod(S_{0,5}^1)\arrow[rr, "\text{cap-off}"] & & \Mod(S_{0,6})
\end{tikzcd}
\end{center}
Here, the map $\Theta_2^1$ is an isomorphism, and the map $\Theta_2$ is surjective with $\ker(\Theta_2)=\la \iota\ra$. The cap-off maps are both given by setting a full-twist about $\del S$ to 0 (geometrically, one may think about capping $\del S_2^1$ with a \emph{marked} disk), and the ``forget" map forgets about the marked point on $S_{2,1}$. The cap-off map forgets about the ``integer part" of the fractional Dehn twist coefficient but preserves the ``fractional part" in each case (i.e. the rotation number of $\psi$ along $\del S$ is the rotation number of the capped-off map $\widehat{\psi}$ at the marked point in the capping disk). See \cite{FM} for more details on the maps involved in this diagram.

Note that the braid $\Delta^2=(\sigma_1\sigma_2\sigma_3\sigma_4)^5\in\Mod(S_{0,5}^1)$ is isotopic to a full-twist about $\del S_{0,5}^1$. It follows that, given a spherical mapping class $f\in\Mod(S_{0,6})$, there are $\zz$-many lifts of $f$ to braids
$$...\Delta^{-4}\beta,~\Delta^{-2}\beta,~\beta,~\Delta^2\beta,~\Delta^4\beta...\in\Mod(S_{0,5}^1)$$
which are all related by powers of $\Delta^2$, and are distinguished by their fractional Dehn twist coefficient $c(\beta)$, as in Theorem \ref{fdtcthm}. For any such $f$, only finitely many such $\beta$ may have braid closure $\widehat{\beta}$ an unknot. This may be seen, for example, from the following theorem of Ito--Kawamuro, which will be a key tool in this paper:

\begin{thm}[Ito--Kawamuro]\label{ikthm}
If the braid closure $\widehat{\beta}$ is an unknot, then $|c(\beta)|<1$.
\end{thm}

Moreover, one may check that $\Delta^2$ lifts to an element of $\SMod(S_2^1)$ which squares to a full twist about $\del S_2^1$. This implies a very useful fact: the lift of $\Delta^2$ is freely isotopic to the hyperelliptic involution $\iota$. One may see this in a number of different ways--- for example, by noting that the full twist about $\del S_2^1$ is freely isotopic to the identity, and that the lift of $\Delta^2$ acts on $H_1(S_2^1)$ by $-\text{id}$.

A genus-two, hyperbolic, fibered knot $K\subset Y$ yields an open book decomposition $(S,h)$ for $Y$, where $S=S_2^1$, and $h\in\Mod(S_2^1)$ is freely isotopic to a pseudo-Anosov $\psi_h:S\to S$. If $\psi_h$ is fixed point free in the interior of $S$, then $h$ is symmetric (i.e. represents an element of $\SMod(S_2^1)$) by \cite{BHS}, so we may think of $h$ as the lift of a 5-braid $\beta\in\Mod(S_{0,5}^1)$. From the perspective of 3-manifolds, this means that $Y$ is the double cover of $S^3$ branched along the braid closure $\widehat{\beta}$. This implies, for example, that if $Y=S^3$ then $\widehat{\beta}$ is the unknot. From the perspective of knots, the original fibered knot $K=\del S_2^1\subset Y$ is the lift of the braid axis $\del S_{0,5}^1\subset S^3$ in the cover.

Baldwin--Hu--Sivek in \cite{BHS} additionally computed the singularity type of $\psi_h$ (and originally observed several of the facts mentioned in the previous paragraph). These observations are recorded as the following result, which is the starting point for many of the ideas in this paper:

\begin{thm}[Baldwin--Hu--Sivek]\label{bhsthm}
Let $K\subset S^3$ be a hyperbolic, genus-two, fibered knot with associated open book decomposition $(S,h)$ satisfying $c(h)\neq 0$. If the pseudo-Anosov representative $\psi_h$ is fixed point free in the interior of $S$, then:
\begin{enumerate}[label=$\bullet$]
    \item $\psi_h$ has singularity type either:
    \begin{enumerate}[label=\hspace{1cm}Case \arabic*:, left=0pt, itemsep=-1pt, topsep=-5pt]
        \item $(6;\emptyset;\emptyset)$, or
        \item $(4;\emptyset;3^2)$
    \end{enumerate}
    \item $h$ is the lift of a 5-braid $\beta$ under the Birman--Hilden correspondence, and
    \item as a knot in $S^3$, the braid closure $\widehat{\beta}$ is the unknot.
\end{enumerate}
\end{thm}

Theorem \ref{bhsthm} yields strong constraints on the braid $\beta$. For example, because $h$ has pseudo-Anosov mapping class, we know that $\beta$ does, too, and we may determine the singularity type of its pseudo-Anosov representative $\psi_\beta$ from that of $\psi_h$, by appealing to Lemma 3.7 of \cite{BHS}. If $\psi_h$ has singularity type $(6;\emptyset;\emptyset)$ then $\psi_\beta$ has singularity type $(3;1^5;\emptyset)$. And, if $\psi_h$ has singularity type $(4;\emptyset;3^2)$ then $\psi_\beta$ has singularity type $(2;1^5;3)$.

\subsection{Fibered surfaces and train tracks}\label{subsec:fibsurf}

For the remainder of the paper, we will denote by $S'$ the surface $S$ with its marked points deleted, and by $\widehat{S}$ the closed surface obtained by capping-off the boundary components of $S$ with disks and marking a point in the interior of each disk. We will also assume that the surface $S'$ has negative Euler characteristic.

In \cite{BH}, Bestvina and Handel prove that one may associate to any geometric $\psi$ a \textit{fibered surface} $F \subseteq S'$. This fibered surface is decomposed into \textit{strips} and \textit{junctions}, where the strips are foliated by intervals, i.e. \textit{leaves}. See Figure \ref{fig:Decomp}. Together, the leaves and junctions of $F$ are called \textit{decomposition elements}, and $\psi(F) \subseteq F$, sending decomposition elements into decomposition elements and, in particular, junctions into junctions. Collapsing each decomposition element to a point produces a graph $G$ with a graph map $g: G \to G$. The vertices of $G$ correspond to the junctions of $F$, and the edges of $G$ correspond to the strips of $F$. 

\begin{figure}
    \centering
    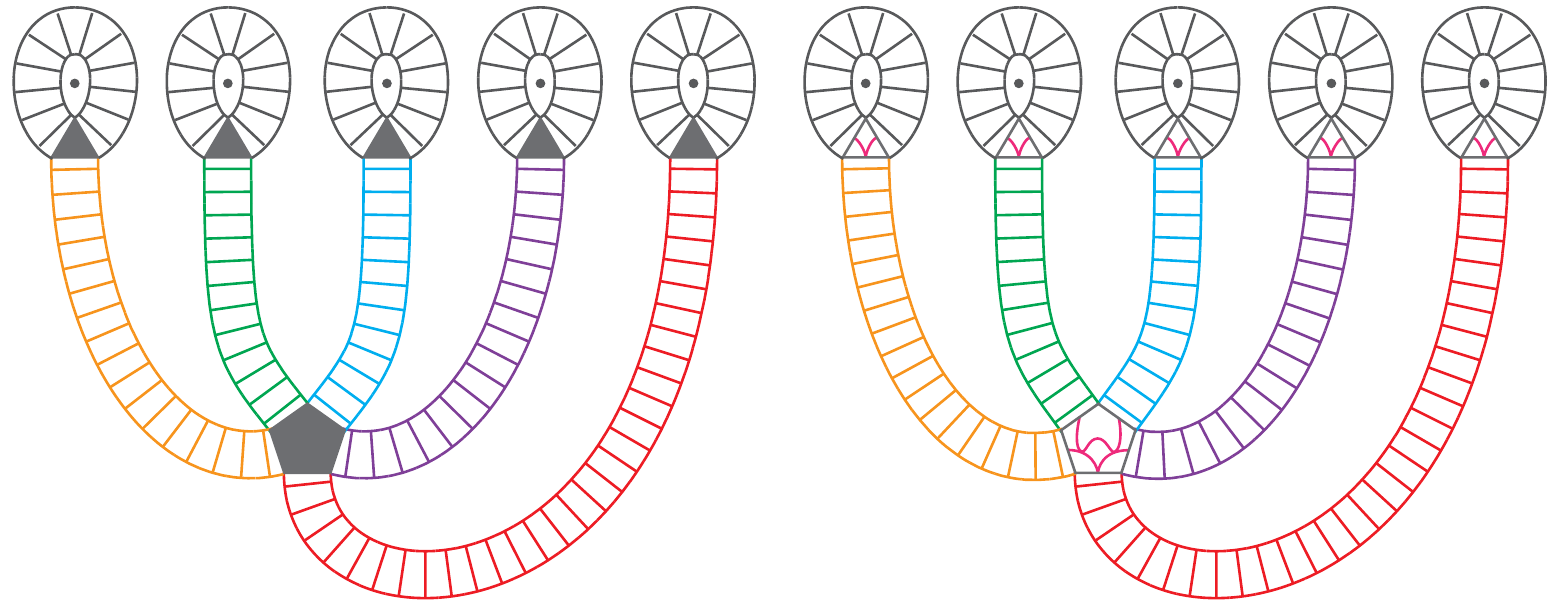
    \caption{Left: the fibered surface for some geometric $\psi$ on $S_{0,5}^1$. The shaded regions are the junctions, and the striped bands connecting them are the strips. Right: following Bestvina-Handel, one inserts additional, ``infinitesimal" edges into the junctions. These will also be inserted into the graph $G$ that one obtains by collapsing all of the decomposition elements. Their inclusion will produce the smooth analog $\tau$ of $G$, which is a train track. See Figure \ref{fig:trackex} below.}
    \label{fig:Decomp}
\end{figure}

Roughly speaking, a graph map $g$ is \textit{efficient} if the image of no edge backtracks under any power of $g$. After adjusting $F$ so that $g$ is efficient, Bestvina and Handel construct a ``smoothed" version of $G$ as follows. Within each junction $J \subseteq F$, one inserts additional edges that smoothly connect the strips of $F$ and encode how images of strips under $\psi$ pass through $J$.

In this way, we obtain a new graph $\tau$ smoothly embedded in the punctured surface $S'$, called a \textit{train track}. At each vertex $s$ of $\tau$, called a \textit{switch}, there is a well-defined tangent line. Two arcs $a, b$ of $\tau$ are \textit{tangent} at $s$ if $a(0)=b(0)=s$ and $a'(0)=b'(0)$. A \textit{cusp} is the data of a pair $(a,b)$ of adjacent arcs tangent at $s$. See Figure \ref{fig:trackex} for an example.

The following proposition appears as Proposition 3.3.5 in \cite{BH}.

\begin{prop}\label{Prop:track}
Suppose $\psi$ is pseudo-Anosov. Then in the capped surface, each component of $\widehat{S} \setminus \tau$ is either:
\begin{enumerate}
\item a disk with $k \geq 3$ cusps on its boundary, or
\item a disk with a single marked point in its interior and $k \geq 1$ cusps on its boundary.
\end{enumerate}
\end{prop}
\begin{rem}\label{cuspsingrem}
From this perspective, the cusps of a component $C\subset\widehat{S}\setminus\tau$ correspond precisely to the prongs of a singularity $p\in C$ of the invariant foliations of $\psi$.
\end{rem}

\begin{figure}
    \centering
    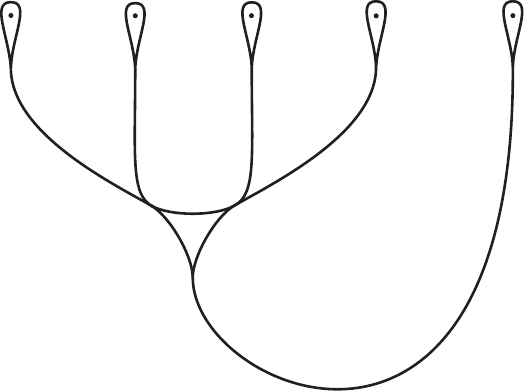
    \caption{A train track $\tau$ on the five-punctured disk $D_5$. The components of the complement of $\tau$ consist of: five once-punctured \textit{monogons}, i.e. disks with a single boundary cusp; a \textit{trigon}, i.e. a disk with three boundary cusps; and an exterior once-punctured bigon. Pseudo-Anosovs carried by this track lie in the stratum $(2;1^5;3)$.}
    \label{fig:trackex}
\end{figure}

\begin{defn}
An \textit{edge path} in $\tau$ is a map $e: I \to \tau$ such that $e(0)$ and $e(1)$ are switches. A \textit{train path} is an edge path that is also a smooth immersion. The \textit{length} of a train path $e$ is defined to be the number of edges traversed by $e(I)$, counting with multiplicity. Let $e(I)=e_1 \cdots e_k$ denote a train path whose directed image traverses first $e_1$, then $e_2$, etc. See Figure \ref{fig:TrainPaths} for examples.
\end{defn}

\begin{figure}
    \centering
    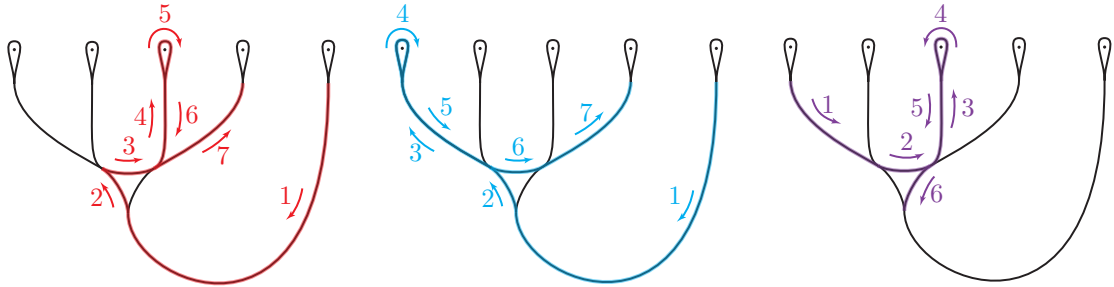
    \caption{Three edge-paths on the train track $\tau$ from Figure \ref{fig:trackex}. \textbf{Left:} an edge path of length 7 which is not a train path, since it makes several sharp turns. \textbf{Middle:} a train path of length 7 which can be ``pushed off" of $\tau$ into a small neighborhood so that it does not intersect itself. \textbf{Right:} a train path of length 6 which cannot be ``pushed off" of $\tau$ so that it becomes injective.}
    \label{fig:TrainPaths}
\end{figure}

\begin{defn}
A \textit{train track map} is a map $f: \tau \to \tau$ such that for any train path $g: I \to \tau$ the composition $f \circ g: I \to \tau$ is a train path.
\end{defn}

\begin{rem}
Note that if $f: \tau \to \tau$ is a train track map, then $f(e)$ is a train path for each edge $e$ of $\tau$. Indeed, from this it follows that $f^k(e)$ is a train path for each $k \geq 1$, and hence $f^k$ is a train track map for all $k \geq 1$. 
\end{rem}

The map $\psi: F \to F$, or equivalently the graph map $g: G \to G$ corresponding to $\psi$ and $F$, defines a map $f: \tau \to \tau$. The fact that $g$ is efficient implies that $f$ is a train track map. In this case, we say that the train track $\tau$ \textit{carries} the map $\psi$, and the map $\psi$ \textit{induces} the train track map $f$. The \textit{data} of a geometric map will then be a triple $(\tau, \psi, f)$ in a commutative diagram:

\[
\begin{tikzcd}
\tau \arrow[r,"\psi"] \arrow{dr}[swap]{f} & \psi(\tau) \arrow[d,"\text{collapse}"]\\
& \tau
\end{tikzcd}
\]

Here, one should imagine $\tau$ being mapped forward by $\psi$ into $F$, meeting the leaves of $F$ transversely. The map $f$ is then defined by collapsing each leaf of $F$ to a point, while inside each junction the arcs of $\psi(\tau)$ are collapsed onto the appropriate edges of $\tau$. See Figure \ref{fig:PAexample}.

\begin{figure}[!h]
    \centering
    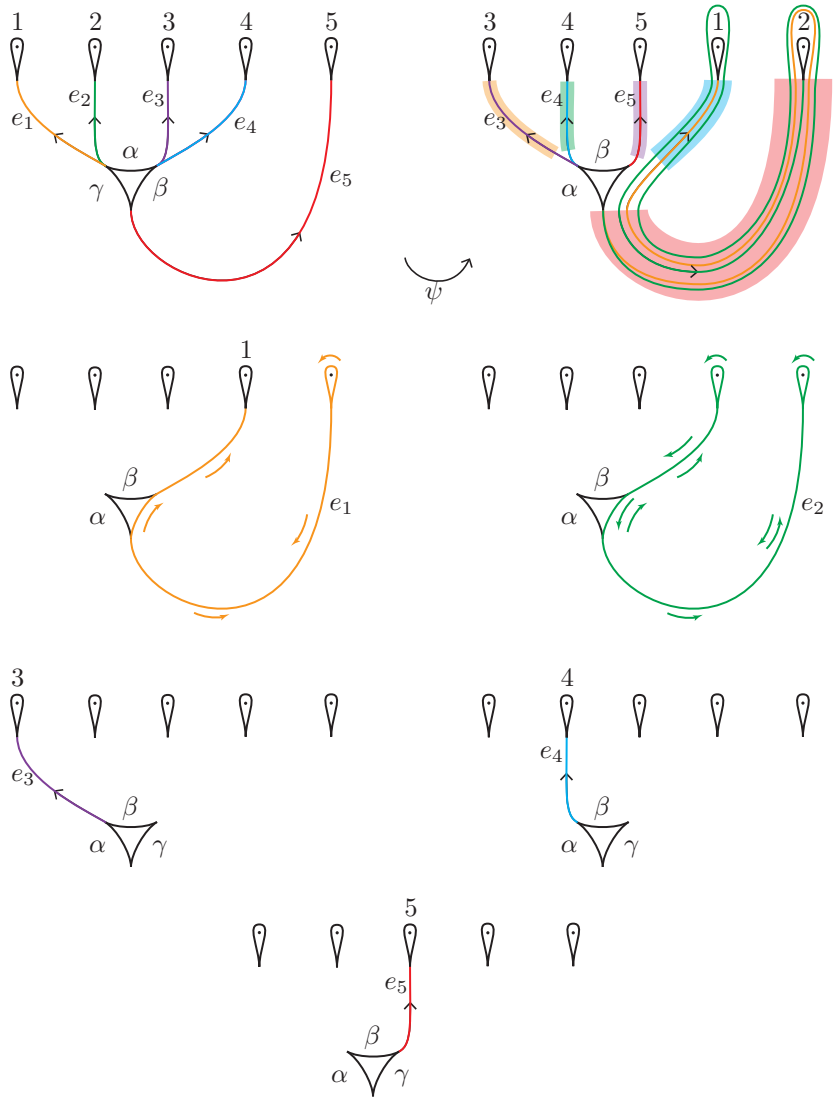
    \caption{\textbf{Top row:} the train track $\tau$ from Figure \ref{fig:trackex} and the action of a pseudo-Anosov $\psi$ that it carries. The real edges of $\tau$ are labeled $e_1, \ldots, e_5$. The shaded regions on the right denote the neighborhoods that deformation retract onto these edges. \textbf{Bottom three rows:} The action of $f=(\text{collapse} \circ \psi)$ on each edge of $\tau$, depicted separately.}
    \label{fig:PAexample}
\end{figure}

The edges of $G$ (other than those loops peripheral to marked points/punctures of $S$) are in bijection with a subset of the edges of $\tau$, which we call the \textit{real} edges. All other edges of $\tau$ are \textit{infinitesimal}. In particular, all edges of $\tau$ contained in a junction of $F$ are infinitesimal. Enumerate the edges of $\tau$ so that $e_1, \ldots, e_k$ are the real edges and $e_{k+1}, \ldots, e_n$ are the infinitesimal edges. For each pair $(i, j)$ with $1 \leq i, j, \leq n$ define the integer
\[
m_{i,j} = \text{the number of times the train path $f(e_j)$ traverses $e_i$.}
\]
\noindent The \textit{extended transition matrix} of $f$ is the matrix $\widetilde{M}$ whose $(i,j)$-entry is the integer $m_{i,j}$. The \textit{transition matrix} of $f$ is the submatrix $M\subset \tilde{M}$ recording the transitions between real edges of $\tau$: in other words,
\[
M=(m_{i,j}) \ \text{where $1 \leq i, j \leq k$.}
\]
\begin{defn}
Let $M$ be a square matrix whose entries are non-negative integers. We say that $M$ is \textit{Perron-Frobenius} if the entries of $M^N$ are strictly positive, for some power $N$. In this case, the Perron-Frobenius theorem states that the eigenvalue of $M$ of largest absolute value is in fact real, simple, and has an eigenvector all of whose entries are positive. We call this eigenvalue the \textit{Perron-Frobenius eigenvalue}, and we say that such an eigenvector is \textit{positive}.
\end{defn}

The next theorem follows from work of Bestvina-Handel in \cite{BH}.

\begin{thm}
Let $(\tau, \psi, f)$ be the data of a geometric map, where $\tau$ satisfies the conclusion of Proposition \ref{Prop:track}. Let $M$ be the transition matrix of $f$. Then 
\[
\text{$\psi$ is pseudo-Anosov} \iff  \text{$M$ is Perron-Frobenius.}
\]
\end{thm}

The Perron-Frobenius eigenvalue $\lambda$ of $M$ is called the \textit{dilatation} of $\psi$. There is a unique right $\lambda$-eigenvector $w$ of $M$, up to scale, and its entries $w_i$ for $i=1 \ldots, k$ define \textit{transverse weights} on the real edges $e_i$ of $M$.

\begin{rem}\label{uniquerem}
If a train track map $F$ is induced by some pseudo-Anosov $\psi$, then such a $\psi$ is unique up to conjugacy in $\Mod(\widehat{S})$. Indeed, Bestvina--Handel in \cite{BH} provide an algorithm to determine the measured foliations preserved by $\psi$. This will be a crucial idea in Section \ref{221section}.
\end{rem}

\subsection{Lifted train track maps and fixed points}\label{liftsec}

\begin{figure}
    \centering
    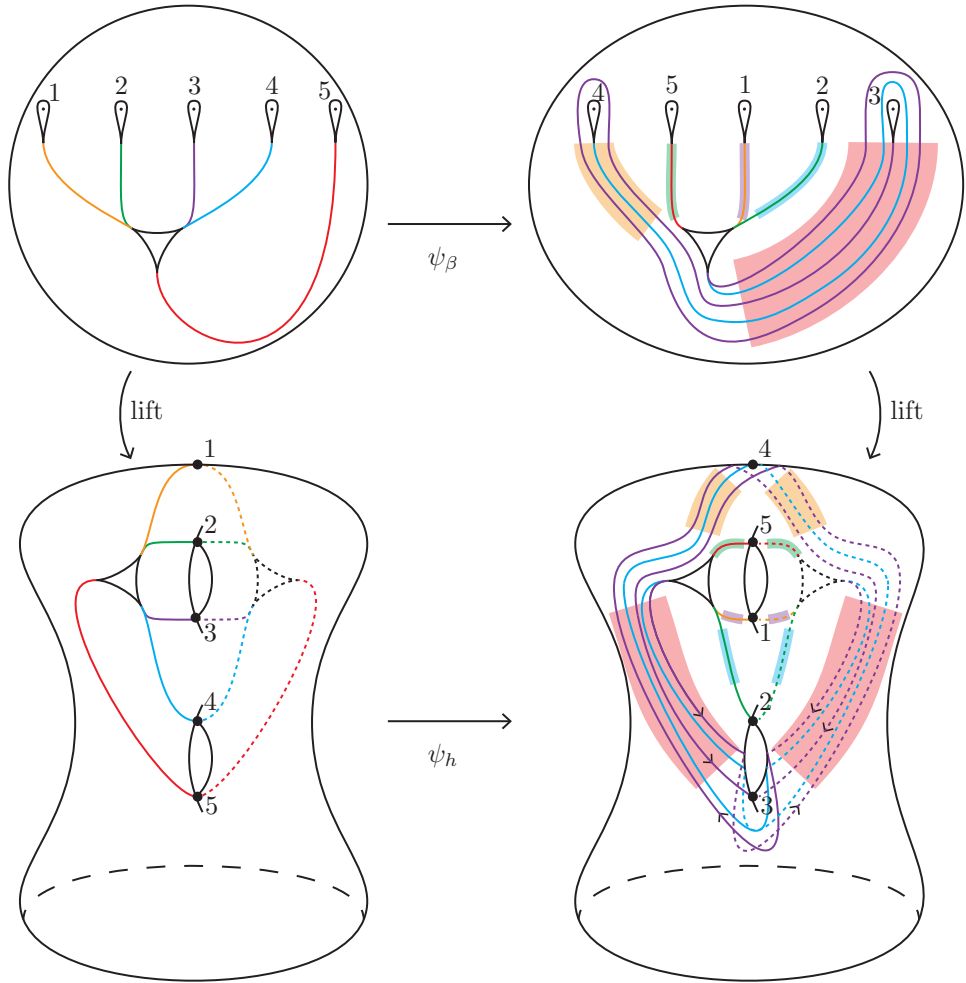
    \caption{Lifting $(\psi_\beta,\tau,f)$ to $(\psi_h,\tilde{\tau},\tilde{f})$. The dotted lines indicate the ``back" of the surface. In this example, $\tau$ is the Peacock track from Figure \ref{fig:TT5} and $\beta$ is conjugate to $\beta_0^{-1}$, from Proposition \ref{braidttprop}. Note here the 1-gons on the disk lift to 2-gons upstairs, which are smoothed out to regular points.}
    \label{liftedtt}
\end{figure}

Let $(\tau, \psi, f)$ be the data of a pseudo-Anosov on $S=S_2^1$. One of the key ideas in this paper is to use the transition matrix $M$ of $f$ to study fixed points of $\psi$ combinatorially. Our main tool to carry out this approach is the following theorem, which follows from work of Los in \cite{Los} and independently Cotton-Clay in \cite{CC}:

\begin{thm}[Los, Cotton-Clay]\label{fpthm}
If $\psi$ is fixed point free in the interior of $S$, then $\text{tr}(M)=0$.
\end{thm}

Our goal will be to use Theorem \ref{fpthm} to restrict the possible train tracks $\tau$ and maps $f:\tau\to\tau$ for a fixed-point-free $\psi$. In the case of genus-two, hyperbolic, fibered knots in $S^3$, the possible types of train tracks $\tau$ are already highly restricted by Theorem \ref{bhsthm} and Remark \ref{cuspsingrem}.

Suppose we start with a genus-two, hyperbolic, fibered knot $K$ with open book decomposition $(S,h)$, where $h$ is the lift of a 5-braid $\beta$ with pseudo-Anosov data $(\psi_\beta, \tau, f)$ on the disk $S_{0,5}^1$. We may lift $\tau$ to a train track $\tilde{\tau}$ on $S$ which carries the pseudo-Anosov representative $\psi_h$ of $h$. As a graph, $\tilde{\tau}$ is constructed by gluing two copies of $\tau$ along the punctures and lifting the infinitesimal $k$-gons around these punctures to infinitesimal $2k$-gons upstairs, as in Figure \ref{liftedtt}.

For this choice of train track $\tilde{\tau}$, we may determine the train track map $\tilde{f}:\tilde{\tau}\to\tilde{\tau}$ induced by $\psi_h$ (so that $h$ has pseudo-Anosov data $(\psi_h,\tilde{\tau},\tilde{f})$) from $f:\tau\to\tau$, up to a binary choice, as follows. For each edge $e\in \tau$, denote by $e^1$ or $e^2$ its two lifts in $\tilde{\tau}$, and write $f(e)=e_1...e_n$ where each $e_i\in\tau$ is an edge. Note that if two edges $e,e'\in\tau$ form a train path $ee'$ then exactly one of the paths $e^ie'^1$ or $e^ie'^2$ is a train path for each $i=1,2$. And, if $e^1e'^i$ is a train path, then so is $e^2e'^j$, where $i\neq j$. Set $\tilde{f_{i_1}}(e^1)=e_1^{i_1}...e_n^{i_n}$, where each $i_j\in\{1,2\}$ and each $e_j^{i_j}e_{j+1}^{i_{j+1}}$ is a train path. Choosing an image for $e^1$ also immediately determines an image for $e^2$, so the maps $\tilde{f_1}$ and $\tilde{f_2}$ are both defined on all of $\tilde{\tau}$. See Figure \ref{liftedtt}.

Note that if $\tilde{f_i}$ is induced by $\psi_i$ for $i=1,2$ then we have $\psi_1=\iota\circ\psi_2$. In particular, at most one of $\tilde{f_1}$ or $\tilde{f_2}$ is induced by $\psi_h$, but this choice may be easily settled by examining $\beta$ as a braid, rather than a mapping class on the punctured sphere. So, we will denote simply by $\tilde{f}$ the well-defined choice of $\tilde{f_i}$ induced by $\psi_h$.

\section{The case with singularity type $(6;\emptyset;\emptyset)$}\label{6section}

The main goal of this section is to prove the following:

\begin{subthm}\label{6case}
Let $K$ be a genus-two, fibered, hyperbolic knot in $S^3$ with associated open book decomposition $(S_2^1,h)$. If $c(h)\neq0$ and the pseudo-Anosov representative $\psi$ of $h$ has singularity type $(6;\emptyset;\emptyset)$, then $K$ is not fixed point free.
\end{subthm}

This resolves Case 1 from the outline in subsection \ref{outline}. Together with Theorem \ref{221thm} in Section \ref{221section}, this will complete the proof of Theorem \ref{fpfthm}. Before turning to the proof of Theorem \ref{6case}, it will be helpful to recall the Lefschetz fixed point theorem, which will be a key ingredient in our proof:

\begin{thm}[Lefschetz fixed point theorem]
Let $S$ be a compact surface and $f: S \to S$ a homeomorphism. Let $f_\ast: H_1(S;\mathbb{Z}) \to H_1(S;\mathbb{Z})$ denote the induced map on first homology. Then

\[
2-\text{tr}(f_\ast) = \sum_{p=f(p)} \Ind(f,p).
\]
\end{thm}

We will apply the Lefschetz fixed point theorem to read off information about the action of a pseudo-Anosov on homology, from its dynamical properties. The relevant result in this vein is an index calculation due to Lanneau--Thiffeault in \cite{LT}:

\begin{prop}[Lanneau--Thiffeault]
Let $\psi: S \to S$ be pseudo-Anosov with orientable invariant foliations, and let $p$ be a fixed $k$-prong singularity of $\psi$. Denote by $\psi_*:H_1(S)\to H_1(S)$ the action on homology, and denote by $\rho(\psi_*)$ the leading eigenvalue of this action, i.e. the eigenvalue with greatest absolute value.
\begin{enumerate}
    \item If $\rho(\psi_\ast)<0$ then $\Ind(\psi,p)=1$; that is, every fixed point of $\psi$ has index $1$.
    \item If $\rho(\psi_\ast)>0$ then either:
    \begin{enumerate}
        \item $\psi$ fixes each prong and $\Ind(\psi,p)=1-k<0$, or
        \item $\psi$ cyclically permutes the prongs and $\Ind(\psi,p)=1$.
    \end{enumerate}
\end{enumerate}
\end{prop}

We can now use Lanneau--Thiffeault's calculation to restrict the dilatation of the pseudo-Anosov representative of a potential fixed-point-free knot $K$ in Theorem \ref{6case}:

\begin{prop}\label{6dil}
Let $K\subset Y$ be a genus-two fixed-point-free knot with $c(K)\not\in\zz$, and suppose that $Y$ is an integer homology sphere. If the pseudo-Anosov representative $\psi$ of $K$ has singularity type $(6;\emptyset;\emptyset)$ then $\psi$ achieves the minimal dilatation $\lambda_2$ among pseudo-Anosovs in genus two.
\end{prop}
\begin{proof}
Let $(S_2^1,h)$ be the open book decomposition of $Y$ associated to $K$, and suppose $\psi$ has no fixed points in the interior of $S$. Because $\psi$ has singularity type $(6;\emptyset;\emptyset)$ by assumption, we may cap-off $\psi$ to a pseudo-Anosov $\widehat{\psi}$ on $S_2$ and extend the foliations preserved by $\psi$ over the capping disk. For this stratum on $S_2$, Masur-Smillie (\cite{MS}) prove that the foliations preserved by $\widehat{\psi}$, and therefore by $\psi$, are necessarily orientable. We will use this fact to apply the Lefschetz fixed point theorem and determine completely the Alexander polynomial of $K$.

Because $K$ is fibered, the Alexander polynomial $\Delta_K$ is equal to the characteristic polynomial $\chi(\psi_*)$ of the action of $\psi$ on homology: $\Delta_K=\chi(\psi_*)$. Because $K$ is a genus-two fibered knot in an integer homology sphere, $\Delta_K$ is a monic, degree-four, palindromic polynomial satisfying $\Delta_K(1)=\pm1$. Moreover, because the fractional Dehn twist coefficient $c(h)\not\in\zz$ by assumption, we know $\widehat{\psi}$ rotates the separatrices of the 6-prong singularity $p$, so that $\Ind(\widehat{\psi},p)=1$ regardless of the sign of $\rho(\psi_*)$. And, because $p$ is the unique fixed point of $\widehat{\psi}$ by assumption, it follows from the Lefschetz fixed point theorem that $\text{tr}(\psi_*)=\text{tr}(\widehat{\psi}_*)=1$.

From the discussion above, we conclude that the coefficients of $t^4$ and $t^0$ in $\Delta_K(t)$ are 1, while the coefficients of $t^3$ and $t$ are $-\text{tr}(\psi_*)=-1$. Now, using the fact that $\Delta_K(1)=\pm1$, we see:
\[
\Delta_K(t) = t^4-t^3 \pm t^2-t+1.
\]
\noindent Because the foliations preserved by $\psi$ are orientable, the dilatation $\lambda(\psi)$ is a root of $\Delta_K$. The polynomial $t^4-t^3+t^2-t+1$, however, has no real roots. We deduce:
\[
\Delta_K(t) = t^4-t^3-t^2-t+1.
\]

\noindent Finally, note that this polynomial has a single root $\lambda_2$ greater than $1$, which is the minimal dilatation achieved by any pseudo-Anosov on the genus-two surface, see e.g. \cite{LT}.
\end{proof}

We will need the following lemma to finish the proof of Proposition \ref{6case}:
\begin{lem}\label{capconj}
Let $h,h'\in\SMod(S_2^1)$ be the lifts of braids $\beta,\beta'$. Suppose that the capped-off maps $\widehat{h}$ and $\widehat{h'}$ on $\widehat{S}$ are conjugate in $\Mod(S_2)$. Then, $\beta$ is conjugate to $\Delta^{2k}\beta'$ for some $k\in\zz$.
\end{lem}
\begin{proof}
Because $\widehat{h}$ and $\widehat{h'}$ are conjugate in $\Mod(S_2)$, and the hyperelliptic involution $\iota$ on $S_2$ is in the center of $\Mod(S_2)$, the conjugating mapping class in $\Mod(S_2)$ descends to the spherical mapping class group $\Mod(S_{0,6})$. It follows that $\beta$ and $\beta'$ are conjugate after capping-off to $\Mod(S_{0,6})$. In particular, $\beta$ is conjugate to $\Delta^{2k}\beta'$ for some $k\in\zz$.
\end{proof}
Though we will not need the following corollary for our purposes, it follows quickly from Lemma \ref{capconj} and we believe it to be helpful in many other contexts, as well.
\begin{cor}\label{liftconj}
Let $h,h'\in\SMod(S_g^r)$ be the lifts of braids $\beta,\beta'$, for $g,r\in\{1,2\}$. Then, $h$ and $h'$ are conjugate in $\Mod(S_g^r)$ if and only if $\beta$ and $\beta'$ are conjugate as braids.
\end{cor}
\begin{proof}
For simplicity, suppose $g=2$ and $r=1$, though the same proof works for the other cases, with minor adjustments. If $\beta$ and $\beta'$ are conjugate, it is clear that $h$ and $h'$ are conjugate, too: we may simply lift the conjugating map to $S_2^1$. On the other hand, suppose $h$ and $h'$ are conjugate in $\Mod(S_2^1)$. It follows that the capped-off maps $\widehat{h}$ and $\widehat{h'}$ are conjugate in $\Mod(S_2)$. Lemma \ref{capconj} now implies that $\beta$ is conjugate to $\Delta^{2k}\beta'$ for some $k\in\zz$. Because $h$ and $h'$ are conjugate in $\Mod(S_2^1)$, we know that $c(h)=c(h')$ (see Theorem \ref{fdtcthm}). It follows that $c(\beta)=2c(h)=2c(h')=c(\beta')$, whereas $c(\Delta^{2k}\beta')=c(\beta')+k$, so we must have that $\beta$ and $\beta'$ are conjugate as braids.
\end{proof}

We need one last result before turning to the proof of Theorem \ref{6case}.

\begin{prop}\label{6braid}
Let $(S_2^1,h)$ be an open book decomposition with $c(h)\not\in\zz$, such that $h$ is symmetric and freely isotopic to a pseudo-Anosov $\psi$ with singularity type $(6;\emptyset;\emptyset)$ and dilatation $\lambda(\psi)=\lambda_2$. Then, $(S_2^1,h)$ is not an open book decomposition for $S^3$.
\end{prop}
\begin{proof}
Because $\psi$ has singularity type $(6;\emptyset;\emptyset)$, we may cap-off $\psi$ to a pseudo-Anosov on $S_2$ with singularity type $(\emptyset;\emptyset;6)$. Lanneau and Thiffeault \cite{LT} show that the pseudo-Anosov on $S_2$ with foliation type $(\emptyset;\emptyset;6)$ and dilatation $\lambda_2$ is unique, up to conjugacy in $\Mod(S_2)$, inverse, and composition with the hyperelliptic involution $\iota$ on $S_2$. Note that the pseudo-Anosov representative $\psi_\alpha$ of the 5-braid $\alpha=\sigma_1\sigma_2\sigma_3\sigma_4\sigma_1\sigma_2$ studied by Ham--Song in \cite{HS} achieves dilatation $\lambda(\psi_\alpha)=\lambda_2$. In particular, the pseudo-Anosov representative $\psi_{A}$ of the lift $A$ of $\alpha$ to $S_2^1$ achieves minimal dilatation $\lambda(\psi_A)=\lambda_2$ with the proper singularity type. It follows that $\widehat{\psi}$ is conjugate in $\text{Mod}(S_2)$ to one of $\widehat{\psi}_A^{\pm1}$ or $\widehat{\psi}_A^{\pm1}\circ\iota$. This further implies that $\widehat{h}$ is conjugate in $\Mod(S_2)$ to one of $\widehat{A}^{\pm1}$ or $\widehat{A}^{\pm1}\circ\iota$.

Because $\iota$ is freely isotopic to the lift of the 5-braid $\Delta^2$ (as described in subsection \ref{BHsection}), $A^{\pm1}\circ\iota$ is freely isotopic to the lift of $\Delta^2\alpha^{\pm1}$. Since $h$ is symmetric by assumption, it is the lift of a braid $\beta$. $A$ is symmetric by construction, so Lemma \ref{capconj} implies that $\beta$ is conjugate as a braid to $\Delta^{2k}\alpha^{\pm1}$ for some $k\in\zz$. In particular, if $(S_2^1,h)$ is an open book decomposition for $S^3$, we can see that $\Delta^{2k}\alpha^{\pm1}$ has unknotted closure for some $k\in\zz$. Moreover, note that the closure of $\Delta^{2k}\alpha$ is unknotted if and only if the closure of $\Delta^{-2k}\alpha^{-1}$ is also unknotted, because the unknot is amphicheiral.

By Theorem \ref{ikthm}, if $\Delta^{2k}\alpha^{\pm1}$ has unknotted closure, we must have $|c(\Delta^{2k}\alpha ^{\pm1})|<1$. We may deduce that $0<c(\alpha)<1$ because $\alpha$ is a positive pseudo-Anosov braid, and $\Delta^{-2}\alpha$ is a negative pseudo-Anosov braid (see Theorem \ref{fdtcthm}). In particular, $k<c(\Delta^{2k}\alpha)< k+1$. So, it suffices to simply check that $\alpha$ and $\Delta^2\alpha^{-1}$ do not have unknotted closure. One may see this in a number of ways--- for example by appealing to the self-linking number: the maximal self-linking number of the unknot is $-1$, but the self-linking numbers $sl(\alpha)=1$ and $sl(\Delta^2\alpha^{-1})=9$ are both positive.
\end{proof}
\begin{proof}[Proof of Theorem \ref{6case}]
Let $K$, $h$, and $\psi$ be as in the statement of the theorem. Recall that since $K$ is fixed point free by assumption, $h$ is symmetric (see Theorem \ref{bhsthm}). Since $(S_2^1,h)$ is an open book decomposition for $S^3$, $|c(h)|\leq 1/2$ (see e.g. \cite{KR}). So if $c(h)\neq 0$ then $c(h)\not\in\zz$. Proposition \ref{6dil} then implies that the dilatation of $\psi$ is $\lambda(\psi)=\lambda_2$, but this contradicts Proposition \ref{6braid}.
\end{proof}
\begin{rem}
Note that our argument does not show that \emph{no} hyperbolic, genus-two, fibered knot in $S^3$ has the Alexander polynomial $t^4-t^3-t^2-t+1$ from the proof of Proposition \ref{6dil}. Rather, any such knot cannot have singularity type $(6;\emptyset;\emptyset)$. Indeed, the knots $11n_{38}$ and $11n_{102}$ on KnotInfo \cite{knotinfo} are genus-two fibered hyperbolic knots with the given Alexander polynomial. But, their singularity types are $(1;\emptyset;3^5)$ and $(2;\emptyset;3^4)$, respectively, so the foliations preserved by their pseudo-Anosov representatives are not orientable. One may check that their dilatations are 1.916... and 2.751..., respectively, both of which are larger than $\lambda_2=1.722...$, and neither of which are roots of the given Alexander polynomial.
\end{rem}

In the stratum $(6;\emptyset;\emptyset)$, we may additionally lift the assumption that $c(h)\neq 0$:
\begin{prop}\label{6case0prop}
Let $K$ be a hyperbolic, genus-two, fibered knot in $S^3$, with associated open book decomposition $(S_2^1,h)$. If $c(h)=0$ and the pseudo-Anosov representative $\psi$ of $h$ has singularity-type $(6;\emptyset;\emptyset)$, then $K$ is not fixed point free.
\end{prop}
\begin{proof}
Suppose that $\psi$ has no fixed points in the interior of $S_2^1$. As in the proof of Theorem \ref{6case}, we may cap-off $\psi$ to a pseudo-Anosov $\widehat{\psi}$ on $S_2$ and extend the foliations preserved by $\psi$. Again, we have that these foliations are orientable. Note that if $\rho(\psi_*)<0$, then an argument identical to that of Theorem \ref{6case} will apply.

So, assuming that $\rho(\psi_*)>0$, the Lefschetz fixed point theorem then yields $\text{tr}(\widehat{\psi}_*)=2-(-5)=7$, because the unique fixed point $p$ given by the boundary 6-prong singularity is unrotated (since $c(h)\in\zz$). Consider the Markov matrix $M$ for a train track representative of $\widehat{\psi}$. It follows from a theorem of Rykken \cite{Ryk} that any eigenvalue of $\hat{\psi}_\ast$ is also an eigenvalue of $M$ (counting multiplicity) except for possibly eigenvalues of $0$ or roots of unity. Note here that a train track representative of $\widehat{\psi}$ has $8$ real edges, so that $M$ is an $8\times 8$ matrix, while $\widehat{\psi}_*$ is $4\times 4$. In particular, $M$ has at most four more eigenvalues than $\widehat{\psi}_*$, and each has absolute value at most one. Hence:
\[
\text{tr}(M) \geq \text{tr}(\widehat{\psi}_\ast) -4=7-4=3.
\]
On the other hand, a well-chosen train track carrying $\widehat{\psi}$ also carries $\psi$. In particular, by Theorem \ref{fpthm}, we can see that $\text{tr}(M)=0$ because $\psi$ is assumed to be fixed point free in the interior of $S$, which is a contradiction.
\end{proof}
\section{The case with singularity type $(4;\emptyset;3^2)$}\label{221section}

\begin{figure}[!ht]
    \centering
\begingroup%
  \makeatletter%
  \providecommand\color[2][]{%
    \errmessage{(Inkscape) Color is used for the text in Inkscape, but the package 'color.sty' is not loaded}%
    \renewcommand\color[2][]{}%
  }%
  \providecommand\transparent[1]{%
    \errmessage{(Inkscape) Transparency is used (non-zero) for the text in Inkscape, but the package 'transparent.sty' is not loaded}%
    \renewcommand\transparent[1]{}%
  }%
  \providecommand\rotatebox[2]{#2}%
  \newcommand*\fsize{\dimexpr\f@size pt\relax}%
  \newcommand*\lineheight[1]{\fontsize{\fsize}{#1\fsize}\selectfont}%
  \ifx\svgwidth\undefined%
    \setlength{\unitlength}{311.37000275bp}%
    \ifx\svgscale\undefined%
      \relax%
    \else%
      \setlength{\unitlength}{\unitlength * \real{\svgscale}}%
    \fi%
  \else%
    \setlength{\unitlength}{\svgwidth}%
  \fi%
  \global\let\svgwidth\undefined%
  \global\let\svgscale\undefined%
  \makeatother%
  \begin{picture}(1,0.35721168)%
    \lineheight{1}%
    \setlength\tabcolsep{0pt}%
    \put(0,0){\includegraphics[width=\unitlength,page=1]{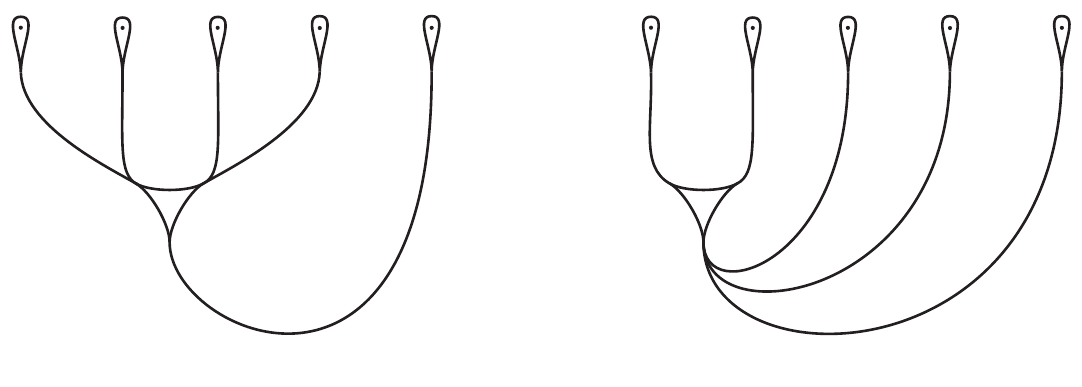}}%
    \put(0.11166779,0.01209173){\color[rgb]{0.1372549,0.12156863,0.1254902}\makebox(0,0)[lt]{\lineheight{1.25}\smash{\begin{tabular}[t]{l}The  Peacock\end{tabular}}}}%
    \put(0.72588881,0.01209173){\color[rgb]{0.1372549,0.12156863,0.1254902}\makebox(0,0)[lt]{\lineheight{1.25}\smash{\begin{tabular}[t]{l}The Snail\end{tabular}}}}%
  \end{picture}%
\endgroup%

    \caption{The two train track classes in the stratum $(2;1^5;3)$ with no joints.}
    \label{fig:TT5}
\end{figure}

The goal of this section is to prove the following theorem, which will resolve Case 2 from the outline in subsection \ref{outline}:

\begin{subthm}\label{221thm}
Let $K$ be a genus-two, fibered, hyperbolic knot in $S^3$ with associated open book decomposition $(S,h)$. If the pseudo-Anosov representative $\psi$ of $h$ has singularity type $(4;\emptyset;3^2)$, then $K$ is not fixed point free.
\end{subthm}
\begin{rem}
Note that we do not need to require $c(h)\neq 0$ here. Between this remark and Proposition \ref{6case0prop}, one may wonder why the assumption $c(h)\neq 0$ is necessary in the statement of Theorem \ref{fpfthm}: it is purely to avoid additional singularity cases. In particular, if $c(h)=0$ then the boundary singularity may be $1$-pronged or $2$-pronged in Theorem \ref{bhsthm}.
\end{rem}

The first step in proving Theorem \ref{221thm} is to observe the following consequence of Theorem \ref{carrythm2}, which we prove at the end of Section \ref{sec:split}:

\begin{thm}\label{thm:221}
Let $\psi$ be a pseudo-Anosov on $S_{0,5}^1$ with singularity type $(2;1^5;3)$. Then $\psi$ is conjugate to a pseudo-Anosov carried by the Peacock train track shown in Figure \ref{fig:TT5}.
\end{thm}

Now, to prove Theorem \ref{221thm}, it suffices by Theorem \ref{thm:221} to look at pseudo-Anosovs carried by the lift of the Peacock track. See Figure \ref{liftedtt} for an image of the lifted track. We will perform a careful analysis of train track maps on this track, together with topological arguments to study a family of braids $\beta_n$ inducing a special collection of train track maps. We present the relevant family of braids $\beta_n$ and their corresponding train track maps in subsection \ref{unknotsubsection}. Then, in subsection \ref{traintracksubsection}, we study train track maps on the Peacock.

\subsection{A family of braids lifting to fixed-point-free maps}\label{unknotsubsection}

\begin{figure}
    \centering
    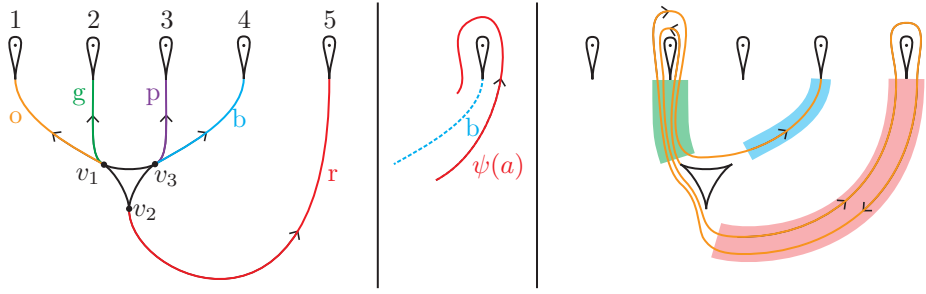
    \caption{\textbf{Left:} The Peacock train track with the labels and orientations we will use in this section. \textbf{Center:} $f(a)$ passes $b$ on the right. \textbf{Right:} $f(o)=g^+r^-g^-b^\circ$.}
    \label{221track}
\end{figure}

For the remainder of this section, $\tau$ will be the Peacock train track depicted on the left in Figure \ref{221track}, with edges and vertices labeled as in the figure (with edges oriented towards the punctures); $\beta$ will be an arbitrary 5-braid with pseudo-Anosov data $(\psi_\beta,\tau,f)$; and $h$ will be the lift of $\beta$ to $S=S_2^1$, with pseudo-Anosov data $(\psi_h,\tilde{\tau},\tilde{f})$, where $\tilde{\tau}$ and $\tilde{f}$ are constructed as in subsection \ref{liftsec}. An image of $\tilde{\tau}$ in this case is shown in Figure \ref{liftedtt}.

Let $a,b$ be any real edges in $\tau$. Note that because each peripheral 1-gon in $\tau$ is adjacent to a unique real edge, the image $f(a)$ is naturally a word of the form $w_1w_2...w_n$, where $w_i\in\{o\bar{o},g\bar{g},p\bar{p},b\bar{b},r\bar{r}\}$ and $w_n\in\{o,g,p,b,r\}$.

\begin{defn}
For real edges $a,b$ of $\tau$, we say that $f(a)$ passes $b$ \emph{on the right} if, before collapsing down to $\tau$, $b$ is to the left of $\psi_\beta(a)$ as in the middle of Figure \ref{221track}. In this case, we write the letter $b^+$ in place of $b\bar{b}$ in the word $f(a)$. We define passing \emph{on the left} analogously, and denote it by $b^-$. If $b$ is the last letter in $f(a)$, we write $b^\circ$. When we allow for multiple possible options, we will write e.g. $b^\pmd$, $b^{+\circ}$, etc. See the right of Figure \ref{221track} for an example.
\end{defn}

Here is the family of braids which we will study:

\begin{figure}
    \centering
    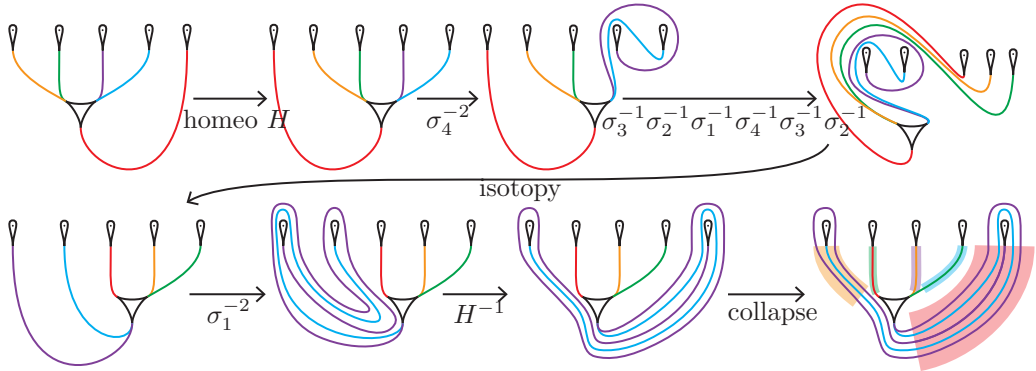
    \caption{The train track map induced by $H\beta_n^{-1}H^{-1}$, where $H$ is an orientation-preserving homeomorphism which swings $r$ around the rest of the track.}
    \label{BnMapfig}
\end{figure}

\begin{prop}\label{braidttprop}
Set $\beta_n=\sigma_1^{n+2}\sigma_2\sigma_3\sigma_4\sigma_1\sigma_2\sigma_3\sigma_4^2$ for $n\geq 0$. Then, $\beta_n^{-1}$ is pseudo-Anosov, and conjugate to a braid carried by $\tau$, which induces the train track map $f_n:\tau\to\tau$ defined by:
\begin{align*}
    f_n(o)&=p^\circ & \hspace{-1cm}f_n(g)&=b^\circ & \hspace{-1cm}f_n(r)&=g^\circ\\
    f_n(p)&=\begin{cases}(r^-o^-)^{(\frac{n}{2}+1)}r^\circ & n \text{ even}\\ (r^-o^-)^{\frac{n+1}{2}}r^-o^\circ & n \text{ odd}\end{cases}
    &\hspace{-1cm}f_n(b)&=\begin{cases}(r^-o^-)^{\frac{n}{2}} r^-o^\circ & n\text{ even}\\ (r^-o^-)^{\frac{n+1}{2}} r^\circ & n\text{ odd}\end{cases} & &
\end{align*}
\end{prop}
\begin{proof}
Figure \ref{BnMapfig} verifies that $H\beta_0^{-1}H^{-1}$ is carried by $\tau$ and induces the train track map $f_0:\tau\to\tau$, where the orientation-preserving homeomorphism $H$ is given by swinging the real edge $r$ around the train track to the other side. For $n\geq 1$, note that $\beta_n=\sigma_1\beta_{n-1}$, and the additional $\sigma_1$ simply adds more twists between the left-most edges before composing with $H^{-1}$ in the last step. This additional twisting adds words of the form $(r^-o^-)$ to $f_n(p)$ and $f_n(b)$, and swaps which edges $p$ and $b$ end on, as in the map in the proposition statement.

It now remains to verify that $\beta_n^{-1}$ is pseudo-Anosov. This can be seen by checking that the transition matrix $M_n$ associated to the train track map $f_n$ is Perron--Frobenius. When determining the matrix $M_n$ from the map $f_n$ given above, it may be helpful to recall that, for each real edge $a,b$ of $\tau$, each instance of $b^\pm$ in $f(a)$ records the word $b\bar{b}$, and each instance of $b^\circ$ records just $b$. Regardless of the parity of $n$, the transition matrix is:

\[
M_n = \begin{pNiceMatrix}[columns-width=auto]
0 & 0 & n+2 & n+1 & 0\\
0 & 0 & 0 & 0 & 1\\
1 & 0 & 0 & 0 & 0\\
0 & 1 & 0 & 0 & 0\\
0 & 0 & n+3 & n+2 & 0
\end{pNiceMatrix}
\]
It is straightforward to check that $M_n^7$ is strictly positive for all $n \geq 0$, so $M_n$ is Perron-Frobenius.
\end{proof}

It follows from Proposition \ref{braidttprop} that any braid inducing the train track map $f_n:\tau\to\tau$ must be conjugate to $\beta_n^{-1}$ within the spherical mapping class group (see Remark \ref{uniquerem}). The following proposition then implies that no braid inducing the map $f_n$ has unknotted closure:

\begin{figure}
    \centering
    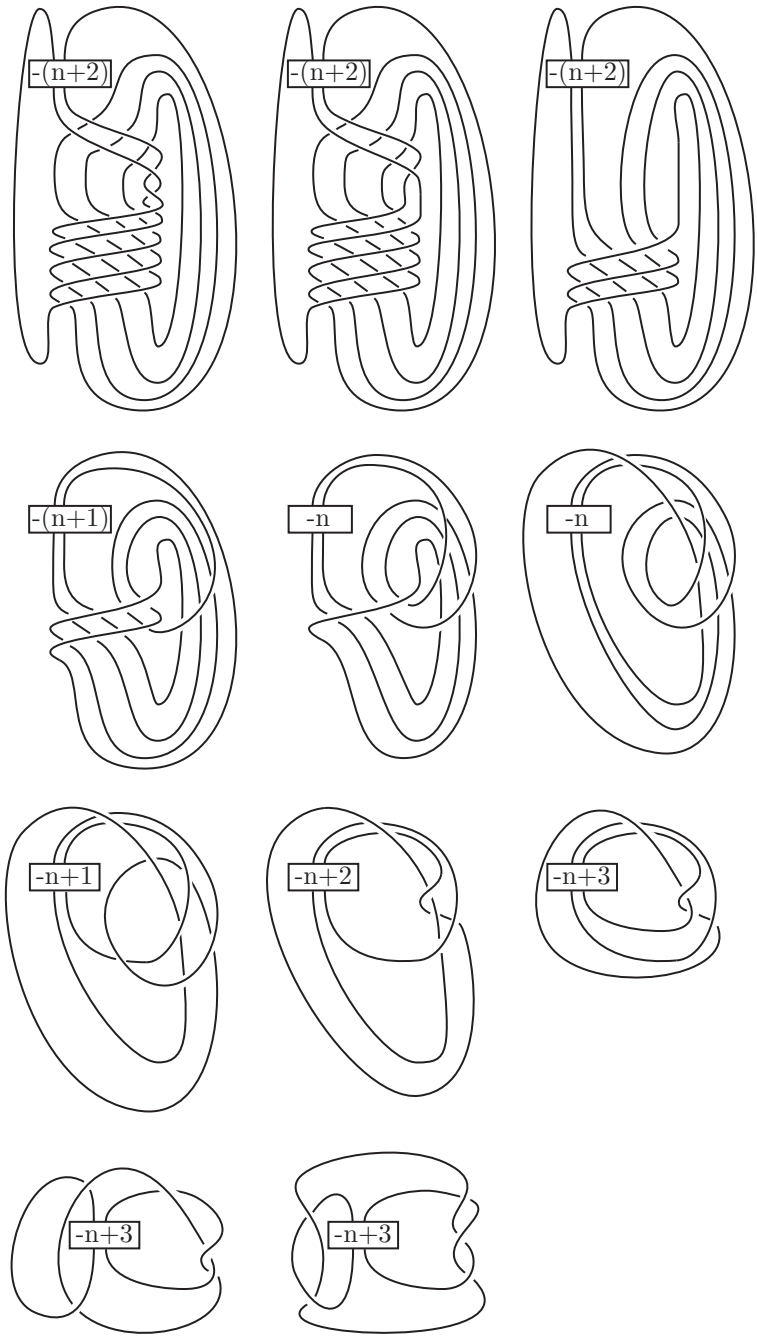
    \caption{Isotopies of $\widehat{\Delta^2\beta_n^{-1}}$ to $P(3,3-n,-2)$}
    \label{isotopies}
\end{figure}

\begin{prop}\label{braidprop}
The braid closure $\widehat{\Delta^{2k}\beta_n^{\pm1}}$ is not an unknot for any $k,n\in\zz$, where $n\geq 0$.
\end{prop}
\begin{proof}
Because $\beta_n$ is a positive pseudo-Anosov braid, we must have $c(\beta_n)>0$. On the other hand, one can check that $\beta_n$ is braid-isotopic to $\beta_n'=(\sigma_1\sigma_2\sigma_3\sigma_4)^2\sigma_4^{n+2}$. One can easily check that $c(\beta_n')\leq 1$ in a variety of ways--- for example by verifying that $\Delta^{-2}\beta_n'$ is $\sigma_1$-negative (see Theorem \ref{fdtcthm}). And, because $\beta_n$ and $\beta_n'$ are braid-isotopic, it follows that $c(\beta_n)\leq 1$, too. So, we have:
$$k<c(\Delta^{2k}\beta_n)\leq k+1$$

Now, by Theorem \ref{ikthm} and the fact that the unknot is amphicheiral, it suffices to check that neither $\widehat{\beta_n}$ nor $\widehat{\Delta^2\beta_n^{-1}}$ is an unknot for any $n\in\zz$. The closure $\widehat{\beta_n}$ is easily seen to be the torus knot $T(2,n+7)$. Figure \ref{isotopies} verifies that the closure $\widehat{\Delta^2\beta_n^{-1}}$ is the 3-stranded pretzel knot $P(3,3-n,-2)$. These knots are all known to not be unknotted.
\end{proof}

\subsection{Train track maps on the Peacock}\label{traintracksubsection}

In this subsection, we retain the notation from the previous subsection. Our remaining goal is to prove the following proposition, which, together with Propositions \ref{braidttprop} and \ref{braidprop}, will imply Theorem \ref{221thm} and complete the proof of Theorem \ref{fpfthm}:

\begin{prop}\label{ttprop}
Let $\psi_\beta$ be a pseudo-Anosov carried by $\tau$, which lifts to a map $\psi_h$ in the cover. If $\psi_h$ is fixed point free, then $\beta$ is conjugate in the spherical mapping class group to $\beta_n$  or $\beta_n^{-1}$ for some $n\in\zz$. In particular, $\beta$ is conjugate as a braid to $\Delta^{2k}\beta_n^{\pm1}$ for some $n,k\in\zz$.
\end{prop}

We begin with some helpful lemmas to simplify the case analysis.

\begin{lem}[Trace Lemma]\label{tracelem}
If $\psi_h$ is fixed point free, then for any real edge $a\in \tau$, we have that $a^\pmd\not\in f(a)$.
\end{lem}
\begin{proof}
First, if $a^\circ\in f(a)$, then the marked point at the end of $a$ is fixed by $\psi_\beta$, and the lift of this marked point is fixed by $\psi_h$. Next, suppose that $a^\pm\in f(a)$, and recall that this means that $f(a)$ contains a word of the form $a\bar{a}$. In the lift, it follows that $\tilde{f}(a^1)$ contains a word of the form $a^i\bar{a}^j$ for some $i,j\in\{1,2\}$ (see the construction in subsection \ref{liftsec}). Because the edges $a^i$ and $a^j$ are not adjacent to infinitesimal loops in the cover (the infinitesimal loops in $\tau$ lift to regular points in the cover), we can see that $i\neq j$. So, $\tilde{f}(a^1)$ contains either $a^1$ or $\bar{a}^1$ as a letter. In either case, the transition matrix of $\tilde{f}$ has non-zero trace, so $\psi_h$ is not fixed point free by Theorem \ref{fpthm}.
\end{proof}

The Trace Lemma also holds in general, with the same proof, for any jointless train track with only 1-pronged punctures. Because we will use the Trace Lemma with great frequency in this section, when we invoke this lemma we will often use the shorthand ``by trace."

\begin{lem}\label{rotationlem}
If $\psi_h$ is fixed point free, then $f(v_i)\neq v_i$ for $i=1,2,3$.
\end{lem}
\begin{proof}
If $f(v_i)=v_i$ for some $i$, then $Df(r)=r^\pmd$, which is forbidden by the Trace Lemma.
\end{proof}

The above lemma implies that $f(v_1)\in\{v_2,v_3\}$, and this choice also determines the images $f(v_i)$ for $i=2,3$. Note that there is a natural horizontal symmetry of $\tau$ induced by reversing the orientation of the disk. Composing with this symmetry takes a braid to its reverse inverse, and a braid lifts to a fixed-point-free map if and only if its reverse inverse does. Hence, it suffices to choose either one of the images $f(v_1)$ as above. Therefore, without loss of generality, $f(v_1)=v_3$.

\begin{figure}
    \centering
    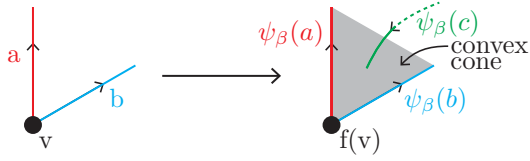
    \caption{$c$ will be absorbed into $f(v)$.}
    \label{fig:absorb}
\end{figure}

\begin{lem}\label{lem:absorb}
Let $a,b$ be any two real edges of $\tau$ with $a$ and $b$ adjacent at initial vertex $v$. Then, for any real edge $c$ of $\tau$, $\psi_\beta(c)$ does not intersect the convex cone determined by the initial segments of $\psi_\beta(a)$ and $\psi_\beta(b)$. See Figure \ref{fig:absorb} for reference.
\end{lem}
\begin{proof}
By injectivity of $\psi_\beta$, we know that $\psi_\beta(c)$ may not cross $\psi_\beta(b)$ or $\psi_\beta(a)$. If $\psi_\beta(c)$ enters the convex cone $X$ on the initial segments of $\psi_\beta(a)$ and $\psi_\beta(b)$, then $\psi_\beta(c)$ must either have its endpoint inside $X$ or must leave $X$. The first case is not possible because $c$ ends at a vertex of an infinitesimal monogon by assumption. Therefore $\psi_\beta(c)$ must enter and leave $X$. Let $A$ and $B$ denote the strips of the fibered surface $F$ that collapse onto $a$ and $b$, respectively. The arc $\psi_\beta(c)$ lies transverse to the fibers of $F$. Assume without loss of generality that $\psi_\beta(c)$ enters $X$ along $A$. Since it must exit $X$, $\psi_\beta(c)$ must subsequently traverse either $A$ or $B$. Neither case is possible, however, since $a$ and $b$ form a cusp: the arc $\psi_\beta(c)$ is forced to be non-smooth.
\end{proof}

When a situation as in Lemma \ref{lem:absorb} arises, we say that the edge $c$ is \emph{absorbed into} $f(v)$. In practice, an edge being absorbed into a vertex is much easier to spot visually than by formal definition: the typical picture is the one depicted in Figure \ref{fig:absorb}.

For the following arguments, recall that we assume without loss of generality that $f(v_1)=v_3$. The figures provided in each proof below are single scenarios appearing in each main case, not exhaustive images of every possibility. We \emph{strongly} encourage the reader to draw by hand the train track maps which are written in words, as they read through each argument.

\begin{lem}\label{yellowlem}
If $\psi_h$ is fixed point free, then $Df(r)\not\in\{o^+,g^+\}$.
\end{lem}
\begin{proof}
If $Df(r)=o^+$, then we must have $f(r)=o^+r^\pmd...$, which is forbidden by the Trace Lemma. The same argument applies to $g^+$.
\end{proof}

\begin{lem}\label{pleftlem}
If $\psi_h$ is fixed point free, then $Df(p)=r^-$.
\end{lem}
\begin{proof}
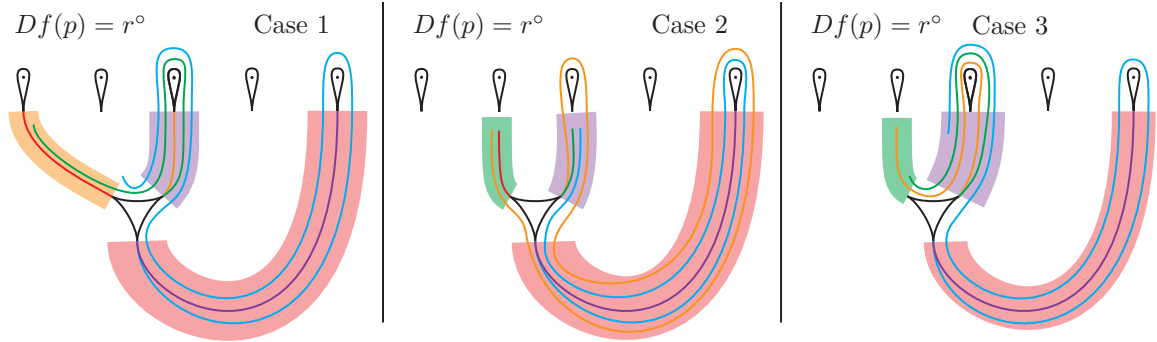
\begin{figure}
    \centering
\begingroup%
  \makeatletter%
  \providecommand\color[2][]{%
    \errmessage{(Inkscape) Color is used for the text in Inkscape, but the package 'color.sty' is not loaded}%
    \renewcommand\color[2][]{}%
  }%
  \providecommand\transparent[1]{%
    \errmessage{(Inkscape) Transparency is used (non-zero) for the text in Inkscape, but the package 'transparent.sty' is not loaded}%
    \renewcommand\transparent[1]{}%
  }%
  \providecommand\rotatebox[2]{#2}%
  \newcommand*\fsize{\dimexpr\f@size pt\relax}%
  \newcommand*\lineheight[1]{\fontsize{\fsize}{#1\fsize}\selectfont}%
  \ifx\svgwidth\undefined%
    \setlength{\unitlength}{525bp}%
    \ifx\svgscale\undefined%
      \relax%
    \else%
      \setlength{\unitlength}{\unitlength * \real{\svgscale}}%
    \fi%
  \else%
    \setlength{\unitlength}{\svgwidth}%
  \fi%
  \global\let\svgwidth\undefined%
  \global\let\svgscale\undefined%
  \makeatother%
  \begin{picture}(1,0.26121428)%
    \lineheight{1}%
    \setlength\tabcolsep{0pt}%
    \put(0,0){\includegraphics[width=\unitlength,page=1]{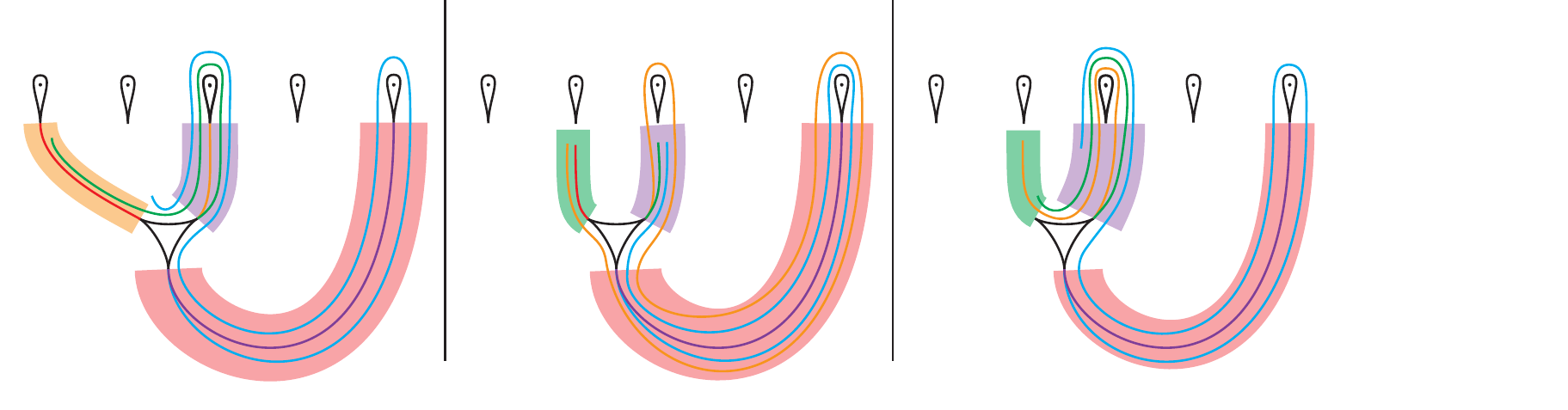}}%
    \put(0.01942857,0.2387){\color[rgb]{0.1372549,0.12156863,0.1254902}\makebox(0,0)[lt]{\lineheight{1.25}\smash{\begin{tabular}[t]{l}$Df(p) = r^\circ$\end{tabular}}}}%
    \put(0.30514286,0.2387){\color[rgb]{0.1372549,0.12156863,0.1254902}\makebox(0,0)[lt]{\lineheight{1.25}\smash{\begin{tabular}[t]{l}$Df(p) = r^\circ$\end{tabular}}}}%
    \put(0.59085714,0.2387){\color[rgb]{0.1372549,0.12156863,0.1254902}\makebox(0,0)[lt]{\lineheight{1.25}\smash{\begin{tabular}[t]{l}$Df(p) = r^\circ$\end{tabular}}}}%
    \put(0.19085714,0.2387){\color[rgb]{0.1372549,0.12156863,0.1254902}\makebox(0,0)[lt]{\lineheight{1.25}\smash{\begin{tabular}[t]{l}Case 1\end{tabular}}}}%
    \put(0.47657143,0.2387){\color[rgb]{0.1372549,0.12156863,0.1254902}\makebox(0,0)[lt]{\lineheight{1.25}\smash{\begin{tabular}[t]{l}Case 2\end{tabular}}}}%
    \put(0.70644286,0.23825714){\color[rgb]{0.1372549,0.12156863,0.1254902}\makebox(0,0)[lt]{\lineheight{1.25}\smash{\begin{tabular}[t]{l}Case 3\end{tabular}}}}%
  \end{picture}%
\endgroup%

    \caption{Cases for the proof of Lemma \ref{pleftlem}.}
    \label{pleftfig}
\end{figure}

First suppose that $Df(p)=r^+$. Then, the second letter in $f(p)$ is either $p$ or $b$. The former is not allowed by trace. In the latter case, note that then $f(b)=r^+b^\pmd...$ since $p$ is to the left of $b$, which is again ruled out by trace.

Now, suppose that $f(p)=r^\circ$. Note that then $Df(b)=r^+$, and this further implies by trace that $f(b)=r^+p^\pmd...$ It then follows that $Df(o)=p^\pmd$, so we will check these three possible cases individually. See Figure \ref{pleftfig}.

\textbf{Case 1:} $f(o)=p^\circ$. In this case, $f(b)=r^+p^+...$ and $Df(g)=p^+$. By trace, it follows that $f(g)=p^+o^\pmd...$, and this in turn forces $Df(r)=o^\pmd$. By Lemma \ref{yellowlem}, we know $Df(r)\neq o^+$, so $Df(r)=o^{-\circ}$. If $Df(r)=o^\circ$, then the real edges $o$, $p$, and $r$ are permuted, implying that the transition matrix of $f$ is not Perron-Frobenius. Finally, if $Df(r)=o^-$, we can see that $f(r)=o^-p^-r^\pmd...$, which is ruled out by trace.

\textbf{Case 2:} $Df(o)=p^-$. In this case, we must have $f(o)=p^-r^-g^\pmd...$ by trace (or else $f(o)$ is absorbed into either $f(v_3)$ or $f(v_1)$, if it goes ``inside" $b$ or $g$, respectively). This further implies that $f(g)=p^-r^-g^\pmd...$, which is ruled out by trace.

\textbf{Case 3:} $Df(o)=p^+$. Here, we have $f(o)=p^+g^\pmd...$, and therefore $f(b)=r^+p^+g^\pmd...$ In particular, this forces $f(g)=p^+g^\pmd...$, which is ruled out by trace.
\end{proof}

\begin{lem}\label{bleftlem}
If $\psi_h$ is fixed point free, then $Df(b)=r^-$.
\end{lem}
\begin{proof}
\begin{figure}
    \centering
    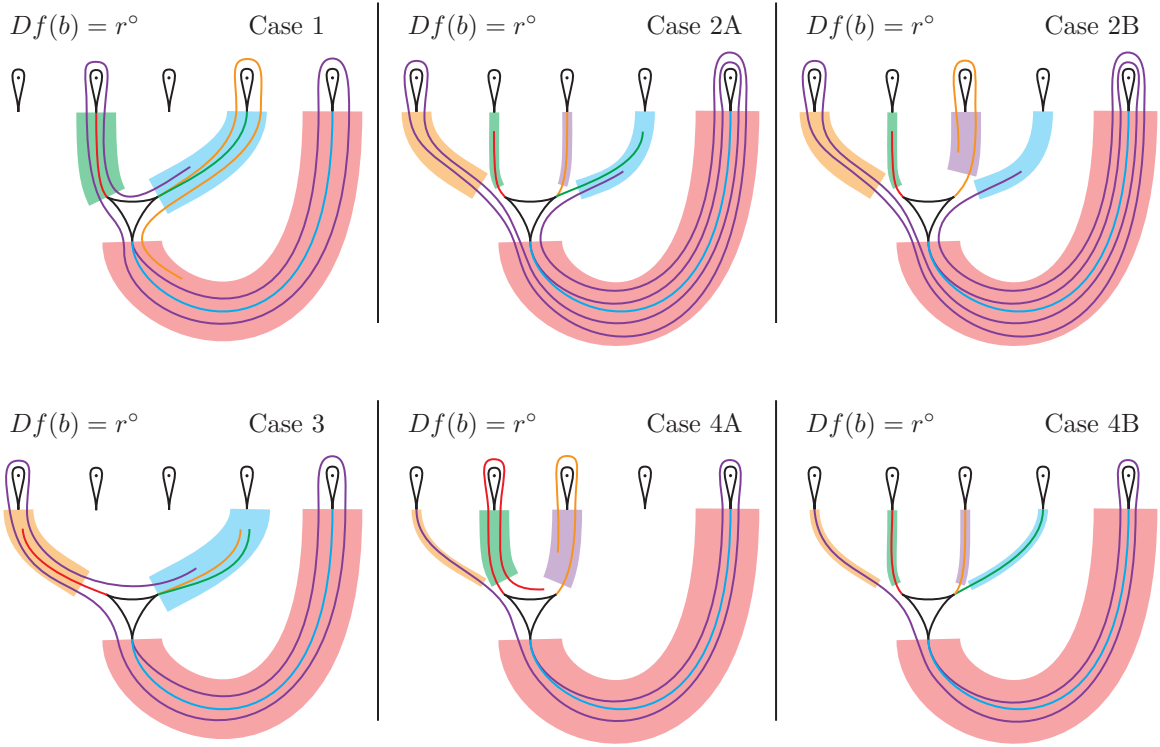
    \caption{Cases for the proof of Lemma \ref{bleftlem}.}
    \label{bleftfig}
\end{figure}

By Lemma \ref{pleftlem}, we may assume $Df(p)=r^-$. Note that $Df(b)=r^+$ is not possible, because then one of $b$ or $p$ will be absorbed into $f(v_3)$.

So, suppose $f(b)=r^\circ$. We branch along cases for the second letter in $f(p)$. Note that if $f(p)=r^-g^{+\circ}...$, then $Df(r)=g^+$, which contradicts Lemma \ref{yellowlem}. So, there are four cases left to consider, shown in Figure \ref{bleftfig}.

\textbf{Case 1:} $f(p)=r^-g^-...$ In this case, we must have $f(p)=r^-g^-b^\pmd...$, because otherwise $p$ is absorbed into $f(v_3)$ (if it follows $r$ next) or contributes trace (if it follows $p$ next). Now note that $Df(o)=b^\pmd$. If $Df(o)=b^{+\circ}$ then $f(g)=b^+g^\pmd...$, which is ruled out by trace. So, $Df(o)=b^-$ and we must have $f(o)=b^-r^-g^\pmd...$ by trace. Note that $Df(g)\neq b^+$ by Lemma \ref{lem:absorb}, so $Df(g)=b^{-\circ}$. If $Df(g)=b^-$, then we have $f(g)=b^-r^-g^\pmd...$, which is ruled out by trace. So, $Df(g)=b^\circ$. Finally, consider cases for $Df(r)$. We know $Df(r)=g^\pmd$, and by Lemma \ref{yellowlem}, we must have $Df(r)=g^{-\circ}$. If $Df(r)=g^\circ$, then the transition matrix is not Perron-Frobenius. If $Df(r)=g^-$, then $f(r)=g^-b^-r^\pmd...$, which is ruled out by trace. 

\textbf{Case 2:} $f(p)=r^-o^+...$ In this case, $f(p)=r^-o^+r^+b^\pmd...$ by trace and, by Lemma \ref{yellowlem}, it follows that $Df(r)=g^{-\circ}$. Now, look at cases for $Df(o)$.

If $Df(o)=b^{+\circ}$ then $f(g)=b^+g^\pmd...$, which is ruled out by trace. And, if $Df(o)=b^-$ then $f(o)=b^-r^-o^\pmd...$, also ruled out by trace. If $Df(o)=p^-$ then similarly $f(o)=p^-r^-o^\pmd...$, which is again ruled out by trace. There are two remaining subcases to consider:

\textbf{Subcase 2A:} $Df(o)=p^\circ$. Here, consider cases for $Df(g)$: either $Df(g)=p^+$ or $Df(g)=b^\pmd$. If $Df(g)=p^+$ then $f(g)=p^+g^\pmd...$ which is ruled out by trace. We cannot have $Df(g)=b^+$ because then $g$ is absorbed into $f(v_1)$. Finally, we cannot have $Df(g)=b^{-\circ}$ because then either $p$ is absorbed into $f(v_1)$ (if $Df(g)=b^\circ$ or $Df(g)=b^-$ with $p$ outside $g$) or $g$ is absorbed into $f(v_3)$ (if $Df(g)=b^-$ and $g$ is outside $p$).

\textbf{Subcase 2B:} $Df(o)=p^+$. Here, consider cases for $Df(r)$: either $Df(r)=g^\circ$ or $Df(r)=g^-$, by Lemma \ref{yellowlem}. If $Df(r)=g^\circ$ then $o$ will be absorbed into $f(v_3)$. If $Df(r)=g^-$, then either $o$ is inside $r$, in which case $f(r)=g^-p^-r^\pmd...$ (which is ruled out by trace); or, $r$ is inside $o$, in which case $o$ will be absorbed into $f(v_3)$.

\textbf{Case 3:} $f(p)=r^-o^-...$ In this case, we must have $f(p)=r^-o^-b^\pmd...$ (otherwise $p$ will be absorbed into $f(v_3)$ or contribute trace), which forces $Df(o)=b^\pmd$ and $Df(g)=b^\pmd$, as well. If $Df(o)=b^+$, then either $f(o)=b^+o^\pmd...$ (which is ruled out by trace), or $f(o)=b^+g^\pmd...$ in which case $f(g)=b^+g^\pmd...$, too (which is again ruled out by trace). And, if $Df(o)=b^-$, then $f(o)=b^-r^-o^\pmd...$, which is ruled out by trace.

So, we must have $f(o)=b^\circ$. Note here that we must have $Df(r)=o^{-\circ}$, by Lemma \ref{yellowlem}, and $Df(g)=b^+$. Now, look at $r$: if $Df(r)=o^-$, then $f(r)=o^-b^-r^\pmd...$ which is ruled out by trace. Finally, if $Df(r)=o^\circ$, then either $g$ is inside $p$, in which case $g$ is absorbed into $f(v_3)$, or $p$ is inside $g$, in which case $p$ will be absorbed into $f(v_1)$.

\textbf{Case 4:} $f(p)=r^-o^\circ$ In this case, note that $Df(r)=g^{-\circ}$ by Lemma \ref{yellowlem}, so consider subcases for $Df(r)$.

\textbf{Subcase 4A:} $Df(r)=g^-$. Here, consider cases for $Df(o)$. If $Df(o)=b^{+\circ}$ then $f(g)=b^+g^\pmd...$, which is ruled out by trace. If $Df(o)=b^-$ then $f(o)=b^-r^-o^-...$, again ruled out by trace. If $Df(o)=p^{-\circ}$ then $f(r)=g^-p^-r^\pmd...$, ruled out by trace. And finally, if $Df(o)=p^+$ then either $o$ is inside $r$ in which case $f(r)=g^-p^-r^\pmd...$ (which is ruled out trace), or $o$ is outside $r$ in which case it will be absorbed into $f(v_3)$.

\textbf{Subcase 4B:} $Df(r)=g^\circ$. Look first at $Df(o)$. If $Df(o)=p^+$ or $Df(o)=b^+$ then $o$ will be absorbed into $f(v_3)$. If $Df(o)=p^-$ or $Df(o)=b^-$ then $f(o)=p^-r^-o^\pmd...$ or $f(o)=b^-r^-o^\pmd...$, both of which are ruled out trace. So, we must have either $Df(o)=p^\circ$ or $Df(o)=b^\circ$ and then it follows that $Df(g)=b^\circ$ or $Df(g)=p^\circ$, respectively, as well, after some simple analysis on $Df(g)$. But, note that $Df(o)=b^\circ$ and $Df(g)=p^\circ$ is not possible, since $\psi_\beta$ is orientation-preserving. And, $Df(o)=p^\circ$ and $Df(g)=b^\circ$ is not possible, because then the transition matrix is not Perron-Frobenius.
\end{proof}

\begin{lem}\label{dfrlem}
If $\psi_h$ is fixed point free, then $Df(r)=g^{-\circ}$.
\end{lem}
\begin{proof}
By Lemma \ref{yellowlem}, we know that $Df(r)\not\in\{g^+,r^+\}$, so we just need to show that $Df(r)\neq o^{-\circ}$. Suppose otherwise, i.e. $Df(r)=o^{-\circ}$. By Lemmas \ref{pleftlem} and \ref{bleftlem}, we know that $Df(p)=Df(b)=r^-$, and then because $Df(r)=o^{-\circ}$ by assumption, it follows that $f(p)=r^-o^-...$ and $f(b)=r^-o^-...$ From here, we must have $f(b)=r^-o^-p^\pmd...$ by trace, but then $f(p)=r^-o^-p^\pmd$, which is ruled out by trace.
\end{proof}

We are finally ready to prove Proposition \ref{ttprop}.

\begin{proof}[Proof of Proposition \ref{ttprop}]

\begin{figure}
    \centering
\begingroup%
  \makeatletter%
  \providecommand\color[2][]{%
    \errmessage{(Inkscape) Color is used for the text in Inkscape, but the package 'color.sty' is not loaded}%
    \renewcommand\color[2][]{}%
  }%
  \providecommand\transparent[1]{%
    \errmessage{(Inkscape) Transparency is used (non-zero) for the text in Inkscape, but the package 'transparent.sty' is not loaded}%
    \renewcommand\transparent[1]{}%
  }%
  \providecommand\rotatebox[2]{#2}%
  \newcommand*\fsize{\dimexpr\f@size pt\relax}%
  \newcommand*\lineheight[1]{\fontsize{\fsize}{#1\fsize}\selectfont}%
  \ifx\svgwidth\undefined%
    \setlength{\unitlength}{450.61500549bp}%
    \ifx\svgscale\undefined%
      \relax%
    \else%
      \setlength{\unitlength}{\unitlength * \real{\svgscale}}%
    \fi%
  \else%
    \setlength{\unitlength}{\svgwidth}%
  \fi%
  \global\let\svgwidth\undefined%
  \global\let\svgscale\undefined%
  \makeatother%
  \begin{picture}(1,0.29243367)%
    \lineheight{1}%
    \setlength\tabcolsep{0pt}%
    \put(0,0){\includegraphics[width=\unitlength,page=1]{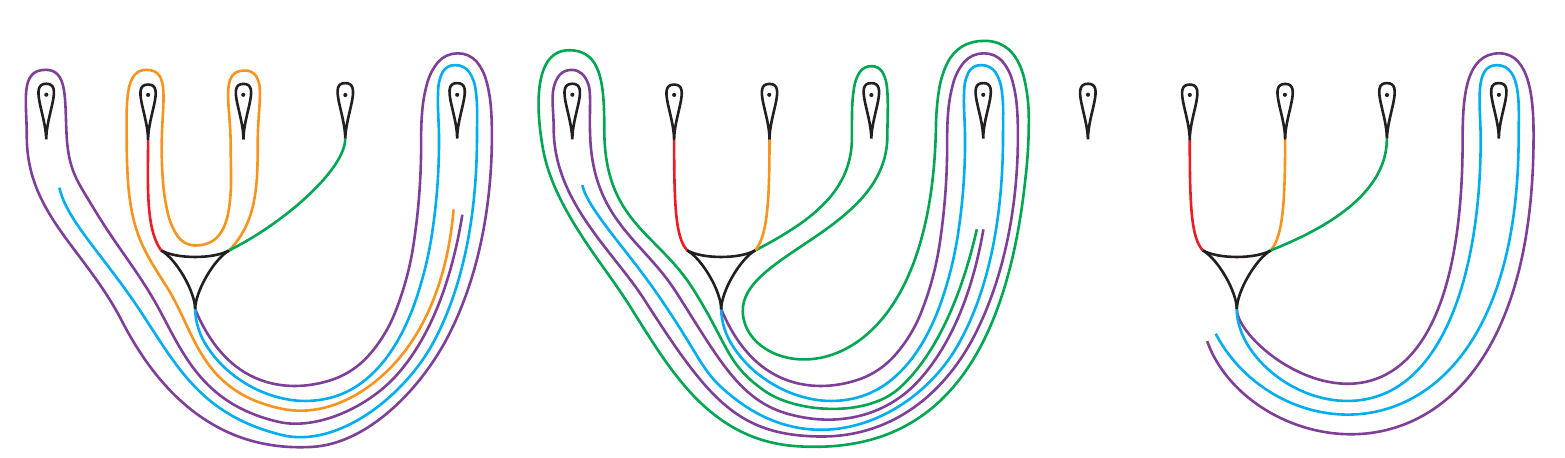}}%
    \put(0.02225292,0.26868279){\color[rgb]{0.1372549,0.12156863,0.1254902}\makebox(0,0)[lt]{\lineheight{1.25}\smash{\begin{tabular}[t]{l}Final\end{tabular}}}}%
    \put(0.35513132,0.26868279){\color[rgb]{0.1372549,0.12156863,0.1254902}\makebox(0,0)[lt]{\lineheight{1.25}\smash{\begin{tabular}[t]{l}Final\end{tabular}}}}%
    \put(0.68800971,0.26868279){\color[rgb]{0.1372549,0.12156863,0.1254902}\makebox(0,0)[lt]{\lineheight{1.25}\smash{\begin{tabular}[t]{l}Final\end{tabular}}}}%
    \put(0.22197996,0.26868279){\color[rgb]{0.1372549,0.12156863,0.1254902}\makebox(0,0)[lt]{\lineheight{1.25}\smash{\begin{tabular}[t]{l}Case 1\end{tabular}}}}%
    \put(0.55485835,0.26868279){\color[rgb]{0.1372549,0.12156863,0.1254902}\makebox(0,0)[lt]{\lineheight{1.25}\smash{\begin{tabular}[t]{l}Case 2A\end{tabular}}}}%
    \put(0.88773675,0.26868279){\color[rgb]{0.1372549,0.12156863,0.1254902}\makebox(0,0)[lt]{\lineheight{1.25}\smash{\begin{tabular}[t]{l}Case 2B\end{tabular}}}}%
    \put(0,0){\includegraphics[width=\unitlength,page=2]{final221cases_svg-tex.pdf}}%
  \end{picture}%
\endgroup%

    \caption{Cases for the proof of Proposition \ref{ttprop}. In this figure, we have chosen to omit the shaded collapsing regions for readability.}
    \label{final221fig}
\end{figure}

Note by the discussion after Lemma \ref{rotationlem}, it suffices to assume $f(v_1)=v_3$. By Lemmas \ref{pleftlem} and \ref{bleftlem}, we know $Df(p)=Df(b)=r^-$, and by Lemma \ref{dfrlem} we know $Df(r)=g^{-\circ}$. From here, we branch along cases for $Df(o)$. The cases $Df(o)= b^\pmd$ can be ruled out quickly as follows.

If $Df(o)=b^+$, then $f(o)=b^+g^\pmd...$ by trace. But, then $f(g)=b^+g^\pmd...$, which is ruled out by trace. If $Df(o)=b^\circ$, then $f(g)=b^+g^\pmd...$ because $Df(r)=g^{-\circ}$, and this is ruled out by trace. Finally, if $Df(o)=b^-$, then $f(o)=b^-r^-g^\pmd...$ by trace. Because $Df(r)=g^{-\circ}$, it follows that $f(p)=r^-g^-...$ and $f(b)=r^-g^-...$ From here, we must have $f(b)=r^-g^-p^\pmd...$ by trace. But, then $f(p)=r^-g^-p^\pmd$, too, which is ruled out by trace.

And, note that if $Df(o)=p^-$, then $f(o)=p^-r^-g^\pmd...$ by trace. Here, we must have $f(p)=r^-g^-p^\pmd...$ because $Df(r)=g^{-\circ}$, which is ruled out by trace. So, we have two cases left to consider, shown in Figure \ref{final221fig}.

\textbf{Case 1:} $Df(o)=p^+$. In this case, we have $f(o)=p^+g^\pmd...$ by trace. If $f(o)=p^+g^{-\circ}...$ then $f(r)$ is either absorbed into $f(v_1)$ or passes over $r$. So $f(o)=p^+g^+r^{\pmd}...$.

Next, consider $Df(g)$. If $Df(g)=b^+$ then $g$ is absorbed into $f(v_1)$, and the case $Df(g)=p^+$ is ruled out quickly by trace. So, we have $Df(g)=b^{-\circ}$. In either case, note that for both of $f(p)$ and $f(b)$, the second letter is in the set $\{o^\pmd, g^-\}$ because $Df(r)=g^{-\circ}$. If $f(p)=r^-g^-...$ then $f(p)=r^-g^-p^\pmd...$, which is ruled out by trace. So, $f(p)=r^-o^\pmd...$

Similarly, if $f(b)=r^-g^-...$ then either $b$ is outside $o$, in which case $b$ will be absorbed into $f(v_1)$, or $o$ is outside $b$, in which case $o$ will be absorbed into $f(v_3)$. So, we must have $f(b)=r^-o^\pmd...$, too.

Now, if $f(p)=r^-o^{+\circ}...$ then $f(b)=r^-o^+r^+b^\pmd...$ which is ruled out by trace. So, we must have $f(p)=r^-o^-r^\pmd...$ And, if $f(p)=r^-o^-r^{-\circ}...$ then $f(o)=b^+g^+r^-o^\pmd...$, which is ruled out by trace. So, we have instead $f(p)=r^-o^-r^+g^-p^\pmd...$ because $Df(r)=g^{-\circ}$, which is again ruled out by trace.

\textbf{Case 2:} $Df(o)=p^\circ$. Here, we branch along subcases for $Df(g)$. Note that if $Df(g)=b^+$ then $g$ is absorbed into $f(v_1)$, and if $Df(g)=p^+$ then $f(g)=p^+g^\pmd...$, which is ruled out by trace. So we have two remaining subcases to consider:

\textbf{Subcase 2A:} $Df(g)=b^-$. Because $Df(r)=g^{-\circ}$, note that if $f(p) = r^- g^-...$ then $p$ is absorbed into $f(v_1)$. A similar argument applies to $f(b)$, so we must have both $f(p)=r^-o^\pmd...$ and $f(b)=r^-o^\pmd...$ Now, if $f(b)=r^-o^+...$ then either $b$ is absorbed into $f(v_3)$, or $f(b)=r^-o^+r^+b^\pmd$, which is ruled out by trace. So, either $f(b)=r^-o^-r^\pmd...$ or $f(b)=r^-o^\circ$.

In the first case, note that if $f(b)=r^-o^-r^+...$, then either $b$ is absorbed into $f(v_3)$, absorbed into $f(v_1)$, or $f(b)=r^-o^-r^+o^+r^+b^\pmd...$, which is ruled out by trace. So, we must have either $f(b)=r^-o^-r^-...$ or $f(b)=r^-o^-r^\circ$.

We can then see that $f(b)=(r^-o^-)^kr^-o^\circ$ or $f(b)=(r^-o^-)^{k+1}r^\circ$ for some $k\geq 0$. The argument that follows will not depend on $k$ (with large $k$, all remaining strands will just turn more times along $r$ and $o$), so for simplicity suppose either $f(b)=r^-o^\circ$ or $f(b)=r^-o^-r^\circ$.

First, suppose $f(b)=r^-o^\circ$. Then, we must have $f(p)=r^-o^-r^\pmd...$ and $f(g)=b^-r^-o^-...$ Note here that if $f(p)=r^-o^-r^{+\circ}...$ then $g$ will be absorbed into $f(v_3)$. But, if $f(p)=r^-o^-r^-...$, then $p$ will be absorbed into $f(v_3)$. This same argument will work for arbitrary $k$ after several twists around $r$ and $o$. 

Next, suppose $f(b)=r^-o^-r^\circ$. The argument is very similar in this case. We must have $f(p)=r^-o^-r^-o^\pmd...$ and $f(g)=b^-r^-o^-r^-o^\pmd...$ If $f(p)=r^-o^-r^-o^{-\circ}...$ then $g$ will be absorbed into $f(v_3)$. And, if $f(p)=r^-o^-r^-o^+$ then either $g$ will be absorbed into $f(v_3)$ or $p$ will be absorbed into $f(v_1)$. As before, the same argument will work for arbitrary $k$ after several additional twists around $r$ and $o$.

\textbf{Subcase 2B:} $Df(g)=b^\circ$. We cannot have $Df(r)=g^-$ since then $r$ will be absorbed into $f(v_1)$. So, we must have $Df(r)=g^\circ$. From here, note that $f(b)=r^-o^{-\circ}...$ because otherwise $b$ will be absorbed into either $f(v_1)$ or $f(v_3)$. In either case, it follows that $f(p)=r^-o^-r^{-\circ}...$ because otherwise $p$ will be absorbed into either $f(v_1)$ or $f(v_3)$. Iterating the same argument, it is now easy to see that $f(b)=(r^-o^-)^{\frac{n}{2}}r^-o^\circ$ or $f(b)=(r^-o^-)^{\frac{n+1}{2}}r^\circ$, and $f(p)=(r^-o^-)^{(\frac{n}{2}+1)}r^\circ$ or $f(p)=(r^-o^-)^{\frac{n+1}{2}}r^-o^\circ$ for some $n\geq 0$.

One may observe that these final train track maps match identically with the ones given in Proposition \ref{braidttprop}. Proposition \ref{braidttprop} and the subsequent discussion then implies that the braid $\beta$ is conjugate in the spherical mapping class group to $\beta_n^{-1}$ for some $n$.
\end{proof}

\section{The tight splitting}\label{sec:split}

This section is devoted to developing a tool which will be integral to the proof of Theorem \ref{thm:221}: a specialized form of ``splitting," which will allow us to restrict our attention to pseudo-Anosovs carried by a single train track.

\subsection{Standardly embedded tracks}

We first describe a particular class of train tracks on the punctured disk, called \textit{standardly embedded} tracks, which will aid in the description of our splitting procedure. Standardly embedded train tracks have previously appeared in the work of Ko--Los--Song (\cite{KLS}), Cho--Ham (\cite{CH}), and Ham--Song (\cite{HS}), who used them to study pseudo-Anosovs on $S_{0,n}^1$ for small $n$.

\begin{defn}\label{defn:stdembed}
An \textit{infinitesimal polygon} of a train track $\tau$ is a connected component of $S_{0,n}^1 \setminus \tau$ whose boundary consists of finitely many infinitesimal edges of $\tau$. A train track $\tau$ on $S_{0,n}^1$ is \textit{standardly embedded} if the following conditions hold:

\begin{enumerate}
\item Every component of $S_{0,n}^1 \setminus \tau$ is an infinitesimal polygon, except for the one containing $\partial S_{0,n}^1$.
\item If two edges of $\tau$ are tangent at a switch, then either both are real or both are infinitesimal.
\item Cusps only occur at vertices of infinitesimal polygons.
\end{enumerate}
\end{defn}

Figure \ref{fig:trackex} is an example of a standardly embedded track, and Figure \ref{fig:PAexample} shows a pseudo-Anosov carried by this track, as well as the induced train track map. Every train track may be adjusted to a standarly embedded one, and this adjustment does not affect which pseudo-Anosovs the track carries. So, we have:

\begin{prop}
Every pseudo-Anosov on $S_{0,n}^1$ is carried by a standardly embedded train track.
\end{prop}

We adapt the following definition from Ham--Song's notion of an elementary folding map \cite{HS}.

\begin{defn}\label{defn:fold}
Let $\tau, \tau_1 \hookrightarrow S_{0,n}^1$ be standardly embedded train tracks. A \textit{Markov map} is a graph map $p: \tau_1 \to \tau$ that maps vertices to vertices, and is locally injective away from the preimages of vertices. An \textit{elementary folding map} is a smooth Markov map such that for exactly one real edge $\alpha$, the image $p(\alpha)$ has word length 2, while the images of all other edges have word length 1. We require that the distinguished edge $\alpha$ belong to a cusp $(\alpha, \beta)$ of $\tau_1$, and that $p(\alpha)$ be of the form

\[
p(\alpha) = p(\beta) \cdot a,
\]

\noindent where $a$ is a real edge joined to $p(\beta)$ by an infinitesimal edge.
\end{defn}

For the purposes of this paper, an elementary folding map $p: \tau_1 \to \tau$ will be the identity map away from the distinguished real edge $\alpha$. See Figure \ref{fig:FoldEx}.

\begin{figure}
    \centering
    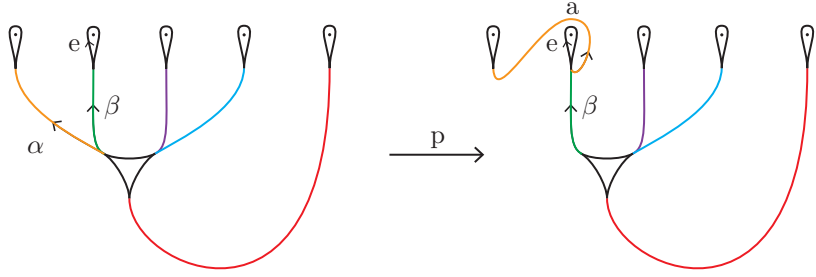
    \caption{An example of an elementary folding map. The map $p$ is the identity except at the edge $\alpha$, which is mapped as a directed path to $\beta \cdot e \cdot a$.}
    \label{fig:FoldEx}
\end{figure}

\begin{rem}
An elementary folding map in our terminology is the composition of two elementary moves in Ham-Song's terminology \cite{HS}.
\end{rem}

\subsection{The tight splitting}

We are ready to introduce a specialized variant of a classical operation on train tracks, known as \textit{splitting} (cf. Section 2.1 of \cite{HP}). Our variant, which we call \textit{tight splitting}, takes as input the data $(\tau, \psi, f)$ of a pseudo-Anosov $\psi$ with an invariant standardly embedded train track $\tau$ and outputs another train track $\tau_1$ that is also invariant under $\psi$. 

More precisely, suppose that $(\tau, \psi, f)$ is the data of a pseudo-Anosov $\psi$ on $S_{0,n}^1$ carried by the standardly embedded train track $\tau$:

\[
\begin{tikzcd}
& \tau\\
\tau \arrow{r}[swap]{\psi} \arrow{ur}{f} & \psi(\tau) \arrow{u}[swap]{\text{collapse}}
\end{tikzcd}
\]

\noindent Suppose further that $\tau_1 \hookrightarrow S_{0,n}^1$ is another standardly embedded train track such that there exists an elementary folding map $p: \tau_1 \to \tau$. Then there is a well-defined elementary folding map $p_\psi: \psi(\tau_1) \to \psi(\tau)$ such that the following diagram commutes:

\[
\begin{tikzcd}
& \tau \\
\tau \arrow{r}[swap]{\psi} \arrow{ur}{f} & \psi(\tau) \arrow{u}[swap]{\text{collapse}}\\
\tau_1 \arrow{r}{\psi} \arrow{u}{p} & \psi(\tau_1) \arrow{u}[swap]{p_\psi}
\end{tikzcd}
\]

\noindent If $\tau_1$ were invariant under $\psi$, we would then be able to complete the above commutative diagram as follows: 

\[
\begin{tikzcd}
& \tau \\
\tau \arrow{r}[swap]{\psi} \arrow{ur}{f} & \psi(\tau) \arrow{u}[swap]{\text{collapse}}\\
\tau_1 \arrow{r}{\psi} \arrow{u}{p} \arrow{dr}[swap]{f_1} & \psi(\tau_1) \arrow{u}[swap]{p_\psi} \arrow{d}{\text{collapse}}\\
& \tau_1
\end{tikzcd}
\]

\noindent Unfortunately, $\tau_1$ need not be invariant under $\psi$, as Example \ref{ex:split_oops} shows. We introduce tight splitting to deal with this problem by taking into account the train track map $f: \tau \to \tau$. In particular, we will show that if $(\tau, \psi, f)$ is the data of a pseudo-Anosov acting on a standardly embedded train track $\tau$, then there is always a split $p: \tau_1 \to \tau$ such that $\tau_1$ is still invariant under $\psi$. See Proposition \ref{tsplitcarry}.

\begin{ex}\label{ex:split_oops}
    Consider the pseudo-Anosov $\psi: S_{0,5}^1 \to S_{0,5}^1$ represented in Figure \ref{fig:split_ex_before}. By computing a right eigenvector for the Perron-Frobenius eigenvalue of the transition matrix $M$, we see that edge $e_2$ has smaller transverse measure than edge $e_3$. Following Harer-Penner, we can perform a split by peeling $e_2$ away from $e_3$ (truthfully, this is two splits since we need to also peel $e_2$ away from the non-expanding edge connecting it to $e_3$). The resulting train track is shown on the left of Figure \ref{fig:split_ex_after}. Unhappily, $\tau_1$ is not invariant under $\psi$.
    
    After some thought, one might notice that a different split does produce an invariant train track for $\psi$. Indeed, we can see from Figure \ref{fig:split_ex_before} that $e_4$ has smaller transverse measure than $e_5$, and performing the corresponding pair of splits (first over the black edge connecting $e_4$ to $e_5$, then over $e_5$) produces an invariant train track $\tau_2$, as desired. See Figure \ref{fig:split_ex_after_2}. This second split differs from the first in that it is compatible with the action of $\psi$ on $\tau$: all paths in the image $\psi(\tau)$ that collapse onto $e_4$ also collapse onto $e_5$. It is this property that allows us to easily isotope strands in $\psi(\tau)$ to lie transverse to the leaves of the fibered neighborhood of $\tau_2$. See Definition \ref{defn:tsplit} and Proposition \ref{tsplitcarry} for more.
\end{ex}

\begin{figure}
    \centering
    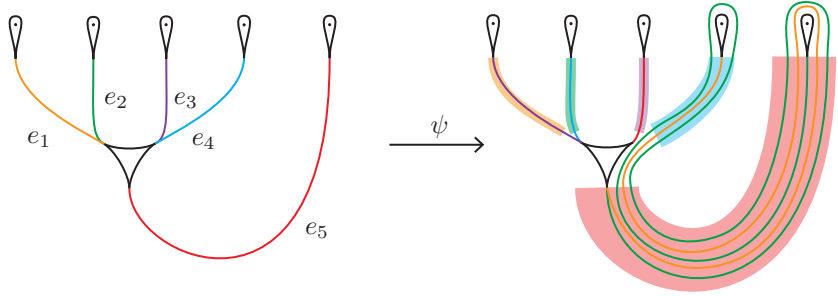
    \caption{The action of a particular pseudo-Anosov $\psi: S_{0,5}^1 \to S_{0,5}^1$ on an invariant train track. Each of the loop edges contains a puncture.}
    \label{fig:split_ex_before}
\end{figure}

\begin{figure}
    \centering
    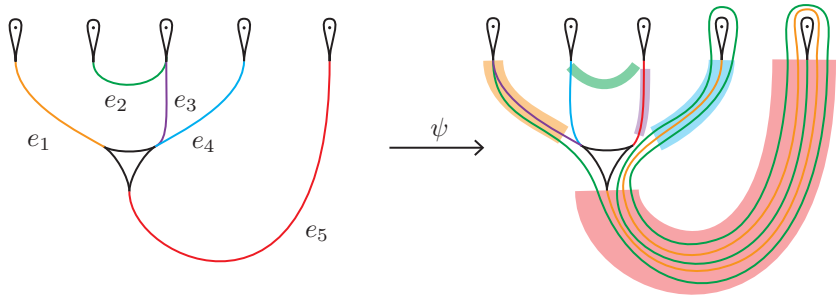
    \caption{The action of $\psi$ after na\"ively splitting $e_2$ over $e_3$. On the left is the new track $\tau_1$, and on the right is the image $\psi(\tau_1)$, up to isotopy. The only difference between the righthand images of this figure and Figure \ref{fig:split_ex_before} is that the image of of $e_2$ has been peeled back along that of $e_3$ and now starts at the leftmost loop. As we can see, $\tau_1$ is not invariant under $\psi$.}
    \label{fig:split_ex_after}
\end{figure}

\begin{figure}
    \centering
\begingroup%
  \makeatletter%
  \providecommand\color[2][]{%
    \errmessage{(Inkscape) Color is used for the text in Inkscape, but the package 'color.sty' is not loaded}%
    \renewcommand\color[2][]{}%
  }%
  \providecommand\transparent[1]{%
    \errmessage{(Inkscape) Transparency is used (non-zero) for the text in Inkscape, but the package 'transparent.sty' is not loaded}%
    \renewcommand\transparent[1]{}%
  }%
  \providecommand\rotatebox[2]{#2}%
  \newcommand*\fsize{\dimexpr\f@size pt\relax}%
  \newcommand*\lineheight[1]{\fontsize{\fsize}{#1\fsize}\selectfont}%
  \ifx\svgwidth\undefined%
    \setlength{\unitlength}{312.84066162bp}%
    \ifx\svgscale\undefined%
      \relax%
    \else%
      \setlength{\unitlength}{\unitlength * \real{\svgscale}}%
    \fi%
  \else%
    \setlength{\unitlength}{\svgwidth}%
  \fi%
  \global\let\svgwidth\undefined%
  \global\let\svgscale\undefined%
  \makeatother%
  \begin{picture}(1,0.35056527)%
    \lineheight{1}%
    \setlength\tabcolsep{0pt}%
    \put(0,0){\includegraphics[width=\unitlength,page=1]{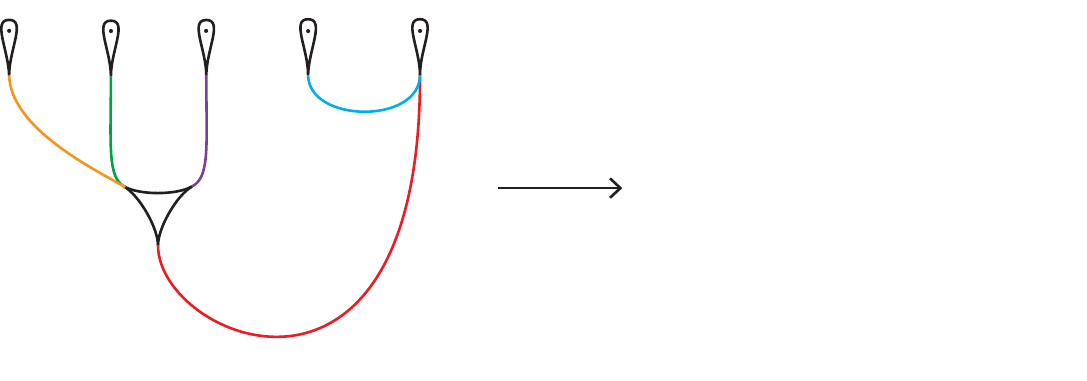}}%
    \put(0.50808025,0.19140277){\color[rgb]{0.1372549,0.12156863,0.1254902}\makebox(0,0)[lt]{\lineheight{1.25}\smash{\begin{tabular}[t]{l}$\psi$\end{tabular}}}}%
    \put(0,0){\includegraphics[width=\unitlength,page=2]{split_ex_after_2_svg-tex.pdf}}%
    \put(0.02246563,0.17920007){\color[rgb]{0.1372549,0.12156863,0.1254902}\makebox(0,0)[lt]{\lineheight{1.25}\smash{\begin{tabular}[t]{l}$e_1$\end{tabular}}}}%
    \put(0.11502872,0.22700396){\color[rgb]{0.1372549,0.12156863,0.1254902}\makebox(0,0)[lt]{\lineheight{1.25}\smash{\begin{tabular}[t]{l}$e_2$\end{tabular}}}}%
    \put(0.19893725,0.22700396){\color[rgb]{0.1372549,0.12156863,0.1254902}\makebox(0,0)[lt]{\lineheight{1.25}\smash{\begin{tabular}[t]{l}$e_3$\end{tabular}}}}%
    \put(0.30279204,0.21767813){\color[rgb]{0.1372549,0.12156863,0.1254902}\makebox(0,0)[lt]{\lineheight{1.25}\smash{\begin{tabular}[t]{l}$e_4$\end{tabular}}}}%
    \put(0.35706887,0.06942374){\color[rgb]{0.1372549,0.12156863,0.1254902}\makebox(0,0)[lt]{\lineheight{1.25}\smash{\begin{tabular}[t]{l}$e_5$\end{tabular}}}}%
  \end{picture}%
\endgroup%

    \caption{The action of $\psi$ after more carefully splitting $\tau$. On the left is the train track $\tau_2$, and on the right is $\psi(\tau_2)$. We obtain a new train track that is still invariant under $\psi$. This is an example of a ``tight split."}
    \label{fig:split_ex_after_2}
\end{figure}

\bigskip

Let $\tau \hookrightarrow S_{0,n}^1$ be standardly embedded, and let $v \in \tau$ be a switch. The \textit{link} of $v$ is the collection $\Lk(v)$ of edges of $\tau$ incident to $v$. The elements of $\Lk(v)$ inherit a natural counterclockwise cyclic order $e_1, \ldots, e_k$. A subset $C \subseteq \Lk(v)$ is \textit{connected} if whenever $e_i, e_j \in C$ and $i<j$, then either

\begin{enumerate}
\item \text{$e_{i+1}, \ldots, e_{j-1} \in C$, or}
\item \text{$e_{j+1}, \ldots, e_k, e_1, \ldots, e_{i-1} \in C$.}
\end{enumerate}

\noindent The collections
\[
R(v) = \{\text{real edges in $\Lk(v)$}\}, \hspace{5mm} I(v) = \{\text{infinitesimal edges in $\Lk(v)$}\}
\]
\noindent are connected. We index the elements of $\Lk(v)$ so that the real edges are $e_1, \ldots, e_m$ under the cyclic order. In other words, from the perspective of $v$ facing its real edges, $e_1$ is the real edge furthest to the right and $e_m$ is the edge furthest to the left.

\begin{defn}
The \textit{right extremal edge} of $v$ is $r(v)=e_1$, and the \textit{left extremal edge} is $l(v)=e_m$. If $R(v)=\{e\}$ is a singleton, then we set $e=l(v)=r(v)$.
\end{defn}

If $v$ is a switch at an infinitesimal loop of $\tau$, we treat each end of the loop as a distinct element of $\Lk(v)$. Hence $I(v)$ always consists of two elements, $i_l$ and $i_r$. These are defined so that, under the cyclic order, we have

\[
l(v) < i_l < i_r < r(v).
\]

\begin{defn}
We denote by $v_l$ the switch of $\tau$ at the other end of $i_l$ from $v$. Similarly, we denote by $v_r$ the switch of $\tau$ at the other end of $i_r$ from $v$. In the case that $v$ is at a loop of $\tau$, we set $v_l=v_r=v$.
\end{defn}

From now on, we set the convention that, for a given switch $v$ of $\tau$, all edges in $R(v)$ are oriented into $v$ as paths.

\begin{defn}\label{defn:tsplit}
Let $\tau \hookrightarrow S_{0,n}^1$ be a standardly embedded train track. Let $v$ be a switch of $\tau$. Fix a train track map $f: \tau \to \tau$. We say that $v$ \textit{splits tightly to the left} or \textit{l-splits} if for every real edge $x \subseteq \tau$ the following two conditions hold:

\begin{enumerate}
    \item Whenever $l(v)$ appears in the train path $f(x)$, it is followed by $\overline{r(v_l)}$, and
    \item whenever $\overline{l(v)}$ appears in the train path $f(x)$, it is preceded by $r(v_l)$.
\end{enumerate}

Similarly, we say that $v$ \textit{splits tightly to the right} or \textit{r-splits} if for every real edge $x \subseteq \tau$ the following two conditions hold:

\begin{enumerate}
    \item Whenever $r(v)$ appears in the train path $f(x)$, it is followed by $\overline{l(v_r)}$, and
    \item whenever $\overline{r(v)}$ appears in the train path $f(x)$, it is preceded by $l(v_r)$.
\end{enumerate}

\noindent In either case, we say that $v$ \textit{splits tightly}. See Figures \ref{fig:Tsplit1} and \ref{fig:Tsplit2}.
\end{defn}

\begin{figure}
    \centering
    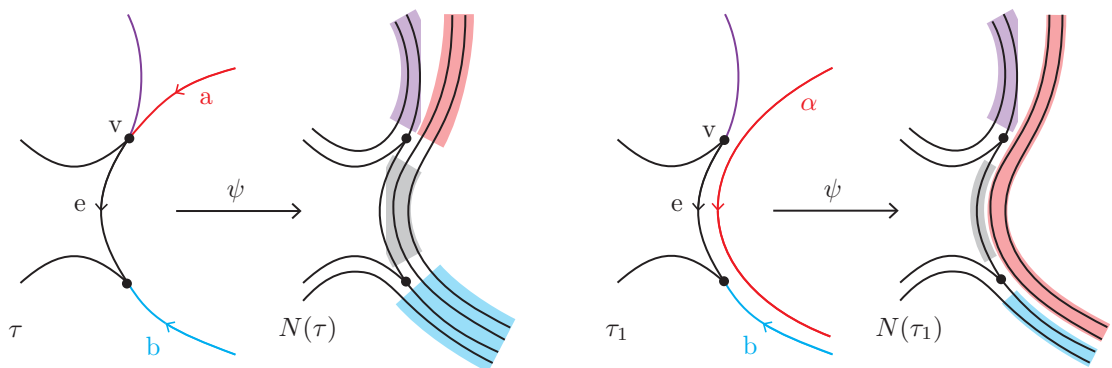
    \caption{\textbf{Left:} part of a train track $\tau$ and the image of a pseudo-Anosov $\psi$ carried by $\tau$. Here $\psi$ induces a train track map $f: \tau \to \tau$ for which $v$ splits tightly to the right. \textbf{Right:} The train track $\tau_1$ after $r$-splitting $v$, and the action of $\psi$ on $\tau_1$. Note in particular that $\psi$ has not changed, only $\tau$ and its fibered neighborhood $N(\tau)$. In each subfigure, the highlighted regions are collapsed by a deformation retraction onto the corresponding edges.}
    \label{fig:Tsplit1}
\end{figure}

\begin{figure}
    \centering
    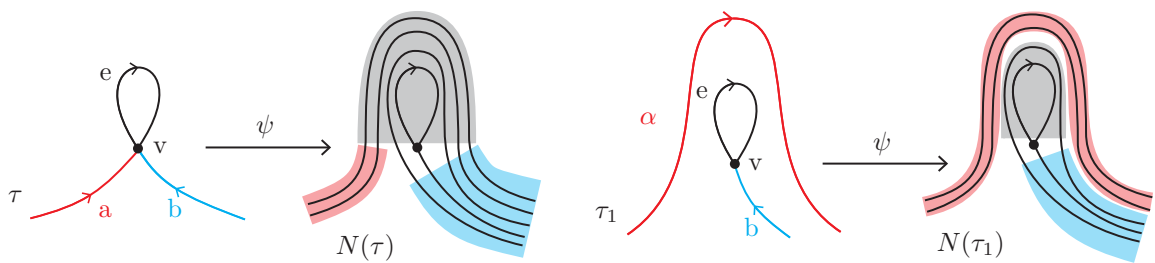
    \caption{\textbf{Left:} another train track $\tau$ and map $f: \tau \to \tau$ for which $v$ splits tightly to the right. \textbf{Right:} the train track $\tau_1$ after $r$-splitting $v$.}
    \label{fig:Tsplit2}
\end{figure}


If $v$ splits tightly, we define a new train track that maps to $\tau$ by an elementary folding map. In this way, we view splitting as an inverse operation to folding. In what follows we will restrict our attention to the case that $v$ tightly splits to the left: all definitions are analogous if $v$ splits tightly to the right. To obtain these analogous statements and proofs, one need only replace all $l$'s with $r$'s and vice versa.

Suppose $v$ $l$-splits. Define $\tau^l_v$ to be the standardly embedded train track obtained by deleting $l(v)$ and replacing it with a real edge $\alpha$ such that

\begin{enumerate}
\item As a directed edge, $\alpha(0)=l(v)(0)$ and $\alpha(1)=r(v_l)(1)$.
\item The edge $\alpha$ forms a bigon (i.e. a two-cusped disk) with the train path $l(v) \cdot \overline{r(v_l})$, and there is an isotopy rel the punctures of $S_{0,n}^1$ so that $\alpha$ lies transverse to the leaves of the fibered neighborhood of $\tau$.
\end{enumerate}

The standardly embedded track $\tau^l_v$ comes equipped with a natural elementary folding map $p: \tau^l_v \to \tau$, defined by

\[
p(x) = \begin{cases}
x & \text{if $x \neq \alpha$}\\
l(v) \cdot \overline{r(v_l)} & \text{if $x = \alpha$.}
\end{cases}
\]

\begin{defn}
If $v$ splits tightly to the left, then the map $p: \tau^l_v \to \tau$ is called a \textit{tight left split} or an \textit{l-split} of $\tau$. We analogously define the \textit{tight right split} or \textit{r-split} $p: \tau^r_v \to \tau$.
\end{defn}

\begin{prop}\label{tsplitcarry}
Suppose $(\tau, \psi, f)$ is the data of a pseudo-Anosov carried by the standardly embedded train track $\tau$:

\[
\begin{tikzcd}
& \tau\\
\tau \arrow{r}[swap]{\psi} \arrow{ur}{f} & \psi(\tau) \arrow{u}[swap]{\text{collapse}}
\end{tikzcd}
\]

\noindent If $v$ $l$-splits, then $\tau^l_v$ carries $\psi$. Hence the above diagram may be completed to the commutative diagram

\[
\begin{tikzcd}
& \tau \\
\tau \arrow{r}[swap]{\psi} \arrow{ur}{f} & \psi(\tau) \arrow{u}[swap]{\text{collapse}}\\
\tau^l_v \arrow{r}{\psi} \arrow{u}{p} \arrow{dr}[swap]{f^l_v} & \psi(\tau^l_v) \arrow{u}[swap]{p_\psi} \arrow{d}{\text{collapse}}\\
& \tau^l_v
\end{tikzcd}
\]

\noindent where $f^l_v$ is a train track map.
\end{prop}

\begin{proof}
Let $F \subseteq S_{0,n}^1$ be a fibered surface for $\psi$ from which the Bestvina-Handel algorithm produces $\tau$. Let $L, I$, and $R$ denote the strips of $F$ collapsing to the (unoriented) edges $l(v), i_l$, and $r(v_l)$ of $\tau$, respectively. Deleting $L$ and replacing it with a strip $A$ collapsing to $\alpha$ produces a new fibered surface $F'$ from which the algorithm produces $\tau^l_v$. The fact that $F'$ is a fibered surface for $\psi$ follows from the fact that $v$ $l$-splits: any strip of $\psi(F)$ passing through $L$ in fact passes through all three of $L, I$, and $R$ in order, and hence after an isotopy we may arrange for the strip to pass through $A$ instead. Furthermore, since $\alpha$ is isotopic to $l(v) \cdot i_l \cdot \overline{r(l_v)}$ and $\psi(L)$, $\psi(I)$, and $\psi(R)$ may be isotoped into $F'$, it follows that $\psi(A)$ may be isotoped into $F'$ as well.
\end{proof}

\begin{prop}\label{tsplitconj}
Suppose that $v$ $l$-splits and let $M$ and $M_v$ be the transition matrices of $f: \tau \to \tau$ and $f^l_v: \tau^l_v \to \tau^l_v$, respectively. Then 

\[
M_v = P^{-1} M P,
\]

\noindent where $P$ is the transition matrix of the elementary folding map $p: \tau^l_v \to \tau$: that is, if $l(v)$ is the $j$th edge and $r(v_l)$ is the $i$th edge, then we have

\[
P = I_n+D_{i,j},
\]

\noindent where $\tau$ has $n$ real edges, $I_n$ is the identity, and $D_{i,j}$ is the square matrix with a 1 in the $(i,j)$-entry and $0$'s elsewhere.
\end{prop}

\begin{proof}
We will argue that we have the following commutative diagram:

\[
\begin{tikzcd}
\tau^l_v \arrow{d}{p} \arrow{r}{f^l_v} & \tau^l_v \arrow{d}{p}\\
\tau \arrow{r}{f} & \tau
\end{tikzcd}
\]

\noindent From this the claim will follow, since each of the arrows is a Markov map, and so upon passing to transition matrices we obtain the relation

\[
PM_v = MP.
\]

Suppose $x$ is an edge of $\tau^l_v$. By the definition of $p$ we have

\[
(f \circ p)(x) = \begin{cases}
f(x) & \text{if $x \neq \alpha$} \\
f(l(v)) \cdot f \left (\overline{r(v_l)} \right ) & \text{if $x =\alpha$.}
\end{cases}
\]

On the other hand, we must understand the map $f^l_v : \tau^l_v \to \tau^l_v$ in order to analyze the composition $p \circ f^l_v$. For any edge $y \in \tau$, define $f'(y)$ to be the word obtained from the train path $f(y)$ by replacing each instance of $l(v) \cdot \overline{r(v_l)}$ with $\alpha$ and each instance of $r(v_l) \cdot \overline{l(v)}$ with $\overline{\alpha}$. In other words, $f'(x)$ is the unique word such that

\[
p(f'(x)) = f(x).
\]

\noindent If $x \neq \alpha$ is an edge of $\tau^l_v$, then $f^l_v(x) = f'(x)$. If $x =\alpha$, then $f^l_v(x)=f^l_v(\alpha) = f'(l(v)) \cdot f'\left (\overline{r(v_l)} \right )$. In either case, we obtain the formula

\[
(p \circ f^l_v)(x) = \begin{cases}
f(x) & \text{if $x \neq \alpha$.}\\
f(l(v)) \cdot f \left (\overline{r(v_l)} \right ) & \text{if $x = \alpha$.}
\end{cases}
\]

\noindent This agrees with the formula for $f \circ p$, so the proof is complete.
\end{proof}

Recall that by the Perron-Frobenius theorem, the dilatation of $\psi$ is a simple eigenvalue of the transition matrix $M$, and there exists a positive right $\lambda$-eigenvector $\mu$ of $M$. For a fixed choice of $\mu$ we will denote by $\mu(x)$ the entry of $\mu$ corresponding to the real edge $x$.

\begin{cor}\label{cor:SplitPos}
Let $(\tau, \psi, f)$ be the data of a pseudo-Anosov carried by a standardly embedded train track. Let $M$ be the transition matrix for $f: \tau \to \tau$, and let $\lambda$ be the dilatation of $f$. Fix a positive right $\lambda$-eigenvector $\mu$ of $M$. If $v$ $l$-splits then $\mu_v=P^{-1} \mu$ is a positive right $\lambda$-eigenvector of $M_v$. Consequently, 
\[
\mu(l(v)) < \mu(r(v_l)).
\]
\end{cor}

\begin{proof}
Since $M_v = P^{-1}MP$, it immediately follows that $\mu_v = P^{-1} \mu$ is a right $\lambda$-eigenvector of $M_v$. At least one entry of $\mu_v$ is positive, since $\mu_v(\alpha)=\mu(l(v))>0$. Therefore $\mu_v$ is positive, since the Perron-Frobenius theorem states that $\lambda$ is a simple eigenvalue of $M_v$ and has a positive eigenvector.

To see that $\mu(l(v)) <\mu(r(v_l))$, observe that
\[
0 < \mu_v(r(v_l)) = \mu(r(v_l)) - \mu(l(v)).
\]
\end{proof}

\begin{ex}

\begin{figure}
\centering
    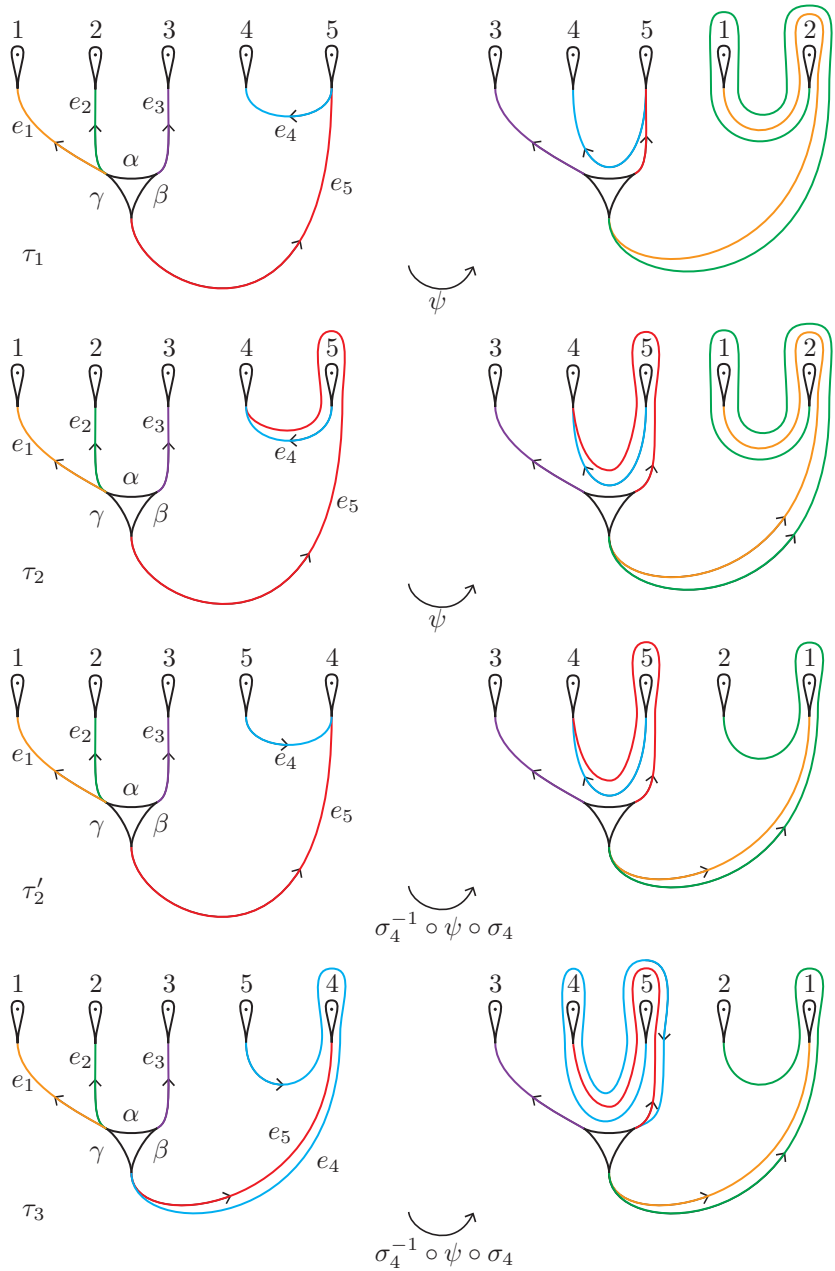
\caption{The track $\tau_1$, $\tau_2$ carries $\psi$. The track $\tau_2'=\sigma_4^{-1}(\tau_2)$ carries $\sigma_4^{-1}\circ\psi\circ\sigma_4$. The track $\tau_3$ carries $\sigma_4^{-1} \circ \psi \circ \sigma_4$.}\label{fig:tight_combined}
\end{figure}

Here is an extended example of a sequence of tight splittings. The maps appearing in this example are closely related to the maps studied in Section \ref{221section}. Let $(\tau, \psi, f)$ be the data of the pseudo-Anosov represented in Figure \ref{fig:tight_combined}. The transition matrix for $f: \tau \to \tau$ is

\[
M _1= \begin{pmatrix}
0 & 0 & 1 & 0 & 0 \\
0 & 0 & 0 & 1 & 0 \\
0 & 0 & 0 & 1 & 1\\
1 & 2 & 0 & 0 & 0\\
1 & 1 & 0 & 0 & 0
\end{pmatrix}.
\]

\noindent The characteristic polynomial of $M_1$ is $\chi(t) = (t+1)(t^4-t^3-t^2-t+1)$. The dilatation of $\psi$ is the root $\lambda$ of this polynomial with largest absolute value, as in subsection \ref{subsec:fibsurf}. A positive right $\lambda$-eigenvector for $M_1$ is 

\[
\mu_1 = \begin{pmatrix}
\mu_1(e_1) \\
\mu_1(e_2) \\
\mu_1(e_3) \\
\mu_1(e_4) \\
\mu_1(e_5)
\end{pmatrix} = \begin{pmatrix}
2+5\lambda-\lambda^2-\lambda^3\\
-2-2\lambda+\lambda^2+\lambda^3\\
1+\lambda+4\lambda^2-2\lambda^3\\
-1-\lambda-\lambda^2+2\lambda^3\\
3\\
\end{pmatrix} = \begin{pmatrix}
2.537...\\
2.628...\\
4.370...\\
4.526...\\
3
\end{pmatrix}
\]

One can see that the vertex at loop 5 splits tightly to the left. Performing this $l$-split produces the track $\tau_2$, which also carries $\psi$. See Figure \ref{fig:tight_combined}. The transition matrix of the $l$-split $p_1: \tau_2 \to \tau_1$ is

\[
P_1 = \begin{pmatrix}
1 & 0 & 0 & 0 & 0 \\
0 & 1 & 0 & 0 & 0 \\
0 & 0 & 1 & 0 & 0 \\
0 & 0 & 0 & 1 & 1 \\
0 & 0 & 0 & 0 & 1
\end{pmatrix} = I_5 + D_{4,5}
\]

\noindent and the transition matrix for $f_2: \tau_2 \to \tau_2$ is

\[
M_2 = P_1^{-1} M_1 P_1 = \begin{pmatrix}
0 & 0 & 1 & 0 & 0 \\
0 & 0 & 0 & 1 & 1 \\
0 & 0 & 0 & 1 & 2 \\
0 & 1 & 0 & 0 & 0 \\
1 & 1 & 0 & 0 & 0
\end{pmatrix}
\]


\noindent which has right $\lambda$-eigenvector

\[
\mu_2 = P_1^{-1}\mu_1 = \begin{pmatrix}
\mu_2(e_1) \\
\mu_2(e_2) \\
\mu_2(e_3) \\
\mu_2(e_4) \\
\mu_2(e_5)
\end{pmatrix} = \begin{pmatrix}
\mu_1(e_1) \\
\mu_1(e_2) \\
\mu_1(e_3) \\
\mu_1(e_4)-\mu_1(e_5)\\
\mu_1(e_5)
\end{pmatrix} = \begin{pmatrix}
2.537...\\
2.628...\\
4.370...\\
1.526...\\
3
\end{pmatrix}
\]

We may conjugate by $\sigma_4^{-1}$ to obtain the track $\tau_2'$, which is slightly easier to read. See Figure \ref{fig:tight_combined}. This move is a standardizing braid move in the language of \cite{KLS}. It is not a tight splitting and is purely cosmetic. It does not alter the transition matrix or any other relevant dynamical information.

We can now see that the switch at loop 4 splits tightly to the right. Performing this $r$-split produces the track $\tau_3$, which also carries $\sigma_4^{-1} \circ \psi \circ \sigma_4$. See Figure \ref{fig:tight_combined}. The transition matrix of the $r$-split $p_2: \tau_3 \to \tau_2$ is 
\[
P_2 = \begin{pmatrix}
1 & 0 & 0 & 0 & 0 \\
0 & 1 & 0 & 0 & 0 \\
0 & 0 & 1 & 0 & 0 \\
0 & 0 & 0 & 1 & 0 \\
0 & 0 & 0 & 1 & 1
\end{pmatrix}=I_5+D_{5,4}
\]
\noindent and the transition matrix for $f_3: \tau_3 \to \tau_3$ is
\[
M_3 = P_2^{-1} M_2 P_2 = \begin{pmatrix}
0 & 0 & 1 & 0 & 0 \\
0 & 0 & 0 & 2 & 1 \\
0 & 0 & 0 & 3 & 2 \\
0 & 1 & 0 & 0 & 0 \\
1 & 0 & 0 & 0 & 0
\end{pmatrix}
\]
\noindent which has right $\lambda$-eigenvector
\[
\mu_3 = P_2^{-1} \mu_2 = \begin{pmatrix}
\mu_2(e_1) \\
\mu_2(e_2) \\
\mu_2(e_3) \\
\mu_2(e_4) \\
\mu_2(e_5)-\mu_2(e_4)
\end{pmatrix} = \begin{pmatrix}
\mu_1(e_1) \\
\mu_1(e_2) \\
\mu_1(e_3) \\
\mu_1(e_4)-\mu_1(e_5)\\
2\mu_1(e_4)-\mu_1(e_4)
\end{pmatrix} = \begin{pmatrix}
2.537...\\
2.638...\\
4.370...\\
1.526...\\
1.473...
\end{pmatrix}
\]
\end{ex}

\subsection{Switch rigidity}

In this section we investigate when a tight splitting is possible at a given switch, identifying the essential obstruction. We call this obstruction \textit{switch rigidity} and show that it is uncommon. Indeed, the orbit of every switch contains a switch that is tightly splittable (cf. Proposition \ref{prop:NoRigidCycles}).

Let $v$ be a switch of the train track $\tau$. Recall that $\Lk(v)$ is the set of edges of $\tau$ incident to $v$. A Markov map $f: \tau \to \tau$ induces a map $Df: \Lk(v) \to \Lk(f(v))$ as follows. Orient all edges in $\Lk(v)$ and $\Lk(f(v))$ away from $v$ and $f(v)$, respectively. Then
\[
Df(a) = b \ \text{if $f(a)$ begins with $b$.}
\]
As a consequence of the Bestvina-Handel algorithm, all elements of $R(v)$ belong to the same \textit{gate}: that is, there exists an integer $k \geq 1$ such that $(Df)^k = D(f^k)$ is constant on $R(v)$. 

\begin{defn}
Let $\tau \hookrightarrow S_{0,n}^1$ be standardly embedded, and let $f: \tau \to \tau$ be a train track map. Let $v$ be a switch of $\tau$ such that $R(v)$ is not a singleton, and set $R(v)=\{e_1, \ldots, e_k\}$. Let $w$ be the switch of $\tau$ such that $f(w)=v$. We say that $v$ is \textit{rigid} if there exist $x_1, \ldots, x_k \in R(w)$ such that
\[
Df(x_i) = e_i \ \text{for all $i$.}
\]
\end{defn}

\begin{figure}
    \centering
    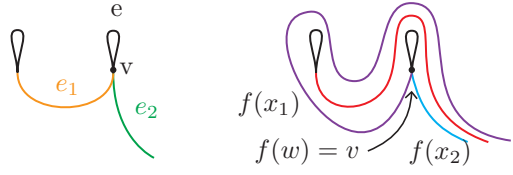
    \caption{An example of a rigid switch. On the left is the switch, on the right the image of the map near the switch.}
    \label{fig:Rigid}
\end{figure}

\begin{lem}\label{lem:connected}
Let $(\tau, \psi, f)$ be the data of the pseudo-Anosov $\psi$ on $S_{0,n}^1$ carried by the standardly embedded $\tau$. Let $w$ be a switch of $\tau$. Write $\alpha=r(w)$, $\beta=l(w)$, and $v=f(w)$. For any $c \in R(v)$ between $Df(\alpha)$ and $Df(\beta)$, there exists $\gamma \in R(w)$ such that $Df(\gamma)=c$. In other words, the set $Df(\Lk(w)) \subseteq \Lk(v)$ is connected.
\end{lem}

\begin{proof}
Suppose $c \in R(v)$ is between $Df(\alpha)$ and $Df(\beta)$. Since $\psi$ is pseudo-Anosov, $f$ is surjective. Hence there exists a real edge $\gamma$ such that $f(\gamma)$ collapses onto $c$. But since $\psi$ is a homeomorphism, $\psi(\gamma)$ cannot intersect $\psi(\alpha \cup \beta)$, so $\gamma$ must be incident to $w$. In other words, $c=Df(\gamma)$. 
\end{proof}

\begin{defn}\label{defn:loopswitch}
We say a switch $v$ of $\tau$ is a \textit{loop switch} if it is incident to an infinitesimal loop.
\end{defn}

The next lemma says that switch rigidity is the only barrier to the existence of a tight splitting at a loop switch. Note that if $v$ is a loop switch, then $v_l=v_r=v$.

\begin{lem}\label{lem:TunlessRigid}
Let $(\tau, \psi, f)$ be the data of a pseudo-Anosov $\psi$ on $S_{0,n}^1$ carried by the standardly embedded $\tau$. Let $v$ be a loop switch, and suppose that $R(v)$ is not a singleton. Then exactly one of the following three possibilities is true.

\begin{enumerate}
\item The switch $v$ splits tightly to the left.
\item The switch $v$ splits tightly to the right.
\item The switch $v$ is rigid.
\end{enumerate}
\end{lem}

\begin{proof}
Let $w$ be the loop switch of $\tau$ such that $f(w)=v$. If either (1) or (2) holds then $v$ cannot be rigid: for example, if $v$ $l$-splits then there does not exist $x \in R(w)$ such that $Df(x)=l(v)$. On the other hand, if $v$ is not rigid then Lemma \ref{lem:connected} implies that at least one of $l(v), r(v)$ is not in the image $Df(\Lk(w))$.

Assume without loss of generality that $l(v) \not \in Df(\Lk(w))$. Then any appearance of $l(v)$ in an image train path is in fact an appearance of $l(v) \cdot \overline{x}$, up to orientation. Here $x$ is some edge in $R(v)$ that might vary. If $x$ is always $r(v)$ then $v$ $l$-splits. Otherwise, we claim that $v$ $r$-splits.

Indeed, suppose that there exists a real edge $y \subseteq \tau$ such that $f(y)$ contains $l(v) \cdot \overline{x}$, up to orientation, for some real edge $x \neq r(v)$. Lemma \ref{lem:connected} implies that $Df(\Lk(w))$ is a subset of the real edges between $l(v)$ and $x$. In particular, $r(v) \not \in Df(\Lk(w))$. Let $z$ be a real edge such that $f(z)$ contains $r(v)$, up to orientation. Since $\psi$ is a homeomorphism and $f(z)$ is a train path, the appearance of $r(v)$ in $f(z)$ must be followed by $\overline{l(v)}$, due to the existence of $\psi(y)$. In other words, $v$ $r$-splits.

Thus we have established that (1) or (2) holds if and only if (3) does not hold. It remains to show that (1) and (2) are mutually exclusive. Corollary \ref{cor:SplitPos} says that if $v$ $l$-splits then $\mu(l(v)) < \mu(r(v))$. It follows that if (1) holds then (2) cannot. The proof is complete.
\end{proof}

The same argument gives the following proposition for a switch not at a loop.

\begin{prop}
Let $(\tau, \psi, f)$ be the data of a pseudo-Anosov $\psi$ on $S_{0,n}^1$ carried by the standardly embedded $\tau$. Let $v$ be a switch of $\tau$, and suppose that $R(v)$ is not a singleton. Suppose additionally that $R(v_l)$ and $R(v_r)$ are singletons. Then at least one of the following three possibilities is true.

\begin{enumerate}
    \item The switch $v$ splits tightly to the left.
    \item The switch $v$ splits tightly to the right.
    \item The switch $v$ is rigid.
\end{enumerate}

\noindent Moreover, case (3) is disjoint from cases (1) and (2).
\end{prop}

Lemma \ref{lem:TunlessRigid} says that if we cannot split at a particular switch $v$, then it is rigid. The natural next step is to consider the preimage switch $v_1$ causing $v$ to be rigid. If $v_1$ is also rigid, we look at its preimage $v_2$. It might happen that we never find a splittable switch. In this case, the periodic orbit of $v$ consists of a cycle of rigid switches.

\begin{defn}
A \textit{rigid cycle} of length $k$ is a collection of rigid switches $v_1, \ldots, v_k \in \tau$ such that $f(v_j)=v_{j-1}$ for all $j$, where the indices are taken modulo $k$.
\end{defn}

\begin{prop}\label{prop:NoRigidCycles}
Rigid cycles do not exist.
\end{prop}

\begin{proof}
Let $v \in \tau$ be a switch. Since $\tau$ is standardly embedded, every element of $R(v)$ belongs to the same gate of $v$, hence there exists $k \geq 1$ such that $(Df)^k$ is constant on $R(v)$. In fact, for all $n \geq k$ we have that $(Df)^n$ is constant on $R(v)$. On the other hand, if $v$ belonged to a rigid cycle of length $n$ then $(Df)^n: R(v) \to R(v)$ would be the identity map, a contradiction.
\end{proof}

\begin{cor}\label{cor:NoRigid}
Let $v \in \tau$ be a switch such that $R(v)$ is not a singleton. Then some iterated preimage switch $w$ of $v$ is not rigid.
\end{cor}

It is well-known that if $(\tau, \psi, f)$ is the data of a pseudo-Anosov, then $f$ permutes the infinitesimal $k$-gons for each $k$ (cf. \cite{BH}). We obtain the following corollary, which will be of central importance in the following section. The \textit{real valence} of a switch $v$ is the cardinality of $R(v)$.

\begin{cor}\label{cor:maxsplit}
Let $n_k$ denote the maximal real valence of a switch at an infinitesimal $k$-gon of $\tau$, where $k \geq 1$. If $n_k > 1$ then there exists a switch of valence $n_k$ at such a $k$-gon which is not rigid. 
\end{cor}

\begin{proof}
The infinitesimal $k$-gons are permuted by $f$. If every such maximal valence switch is rigid, then they must form a rigid cycle, since real valence cannot decrease when passing to the preimage of a rigid switch. This is impossible, since rigid cycles do not exist.
\end{proof}

\subsection{The proofs of Theorems \ref{thm:221}, \ref{carrythm}, and \ref{carrythm2}}\label{splittingsection}\label{sec:enumerate}

In this subsection, we will use the theory of \emph{tight splitting} developed above to prove Theorems \ref{carrythm2}, and see \ref{thm:221} as a consequence. Though Theorem D itself is more general than necessary to prove Theorem \ref{thm:221}, we believe it has wider-reaching applications to surface dynamics.

\begin{defn}
Let $\tau \hookrightarrow S_{0,n}^1$ be a standardly embedded train track. We say a real edge $e$ of $\tau$ is a \textit{stem} if at least one end of $e$ is incident to an infinitesimal $k$-gon, where $k \geq 2$.
\end{defn}

\begin{defn}\label{defn:joint}
Let $\tau \hookrightarrow S_{0,n}^1$ be a standardly embedded train track. We say a loop switch $v \in \tau$ is a \textit{joint} if $|R(v)| \geq 2$.
\end{defn}

\addtocounter{mainthm}{-1}

\begin{mainthm}
Let $\psi$ be a pseudo-Anosov on $S_{0,n}^1$ with at least one $k$-pronged singularity away from the boundary with $k\geq 2$. Then $\psi$ is carried by a train track $\tau$ with no joints.
\end{mainthm}

The central argument in the proof of Theorem \ref{carrythm2} hinges on finding a maximal-valence vertex $v$ near a puncture, and then using Corollary \ref{cor:maxsplit} to tightly split at $v$. Before diving into the proof, we observe one crucial lemma. Although well-known to experts, the authors could not find a complete proof of Lemma \ref{finitelem} in the literature. For the sake of completeness, we have included a proof which arose from a helpful conversation with Karl Winsor.

\begin{lem}\label{finitelem}
For any fixed $n$ and $B>0$, there is a finite number of Perron-Frobenius matrices of size $n$ and spectral radius at most $B$. In particular, there is a finite number of Perron-Frobenius matrices of a given size with a particular Perron-Frobenius eigenvalue.
\end{lem}

\begin{proof}
Fix $n \geq 2$, and let $M$ be an $n \times n$ Perron-Frobenius matrix. Write $M_{i,j}$ for the $(i,j)$th entry of $M$, and $C_j(M)$ for the $j$th column of $M$. An exercise in matrix algebra shows that for each integer $k \geq 1$, 
\[
C_j(M^k) = \sum_{i=1}^n (M^{k-1})_{i, j} \cdot C_i(M).
\]
It is well-known (cf. \cite{W}) that $M^{n^2-2n+2}$ has all entries positive. Hence the smallest column sum of $M^{n^2-2n+3}$ is at least the sum $\norm{M}_1$ of all the entries of $M$. It is not hard to see that the smallest column sum of a Perron-Frobenius matrix is a lower bound on its spectral radius $\rho(M)$. We now have

\[
\rho(M)^{n^2-2n+3} = \rho \left (M^{n^2-2n+3} \right ) \geq \norm{M}_1.
\]

In particular, $\rho(M) \geq \norm{M}_1^{\frac{1}{n^2-2n+3}}$. Since there are only finitely many integer-valued matrices $M$ with $\norm{M}_1$ below a given bound, the result follows.
\end{proof}

\begin{proof}[Proof of Theorem \ref{carrythm2}]
Let $\tau_0 \hookrightarrow S_{0,n}^1$ be a standardly embedded train track carrying $\psi$. We will algorithmically perform a finite sequence of tight splittings on $\tau_0$ to produce the desired track $\tau$ with no joints.

Let $J$ denote the number of cusps at the loop switches of $\tau$, i.e. $J=\sum_v(|R(v)|-1)$, where $v$ ranges over the loop switches of $\tau$. If $J=0$ then there is nothing to prove, so assume $J \geq 1$. By Corollary \ref{cor:maxsplit} there exists a loop switch of $\tau_0$ of maximal valence that can be tightly split. Therefore, we introduce the following simple algorithm.

\begin{enumerate}
    \item Initialize $\tau=\tau_0$ and $\mathcal{M}=\{M_0\}$, where $M_0$ is the transition matrix associated to the data $(\tau_0, \psi, f_0)$.
    \item Find a loop switch of $\tau$ of maximal valence that is not rigid, and split it, obtaining the data $(\tau_1, \psi, f_1)$ with transition matrix $M_1$. Set $\tau=\tau_1$.
    \item If $J$ has decreased by one, return the data $(\tau_1, \psi, f_1)$.
    \item If $J$ has not decreased, add $M_1$ to $\mathcal{M}$ and repeat Steps 2 and 3 with $(\tau_1, \psi, f_1)$.
\end{enumerate}

We claim that this algorithm terminates in finitely many steps, and returns a train track $\tau$ with one fewer joint than $\tau_0$. First, note that splitting at a loop switch $v_0$ either preserves $J$ or decreases it by one. Indeed, let $b$ be the real edge that is split over, i.e. the edge whose transverse weight is reduced (cf. Corollary \ref{cor:SplitPos}). Let $v_b$ denote the switch at the other end of $b$. The tight splitting transfers a cusp from the splitting switch $v_0$ to $v_b$. Thus, in the formula

\[
J = \sum_{\text{$v$ a loop switch}} (|R(v)|-1),
\]

\noindent the contribution from $v_0$ decreases by one, whereas the contribution from $v_b$ either (1) increases by one, if $v_b$ is itself a loop switch; or (2) does not change, if $v_b$ is not a loop switch. In particular, if $b$ is a stem, then splitting over $b$ at the loop switch $v_0$ will always reduce $J$ by one.

It remains to show that, by repeatedly applying the above algorithm, we will eventually split over a stem. Indeed, by Lemma \ref{finitelem} there are only finitely many possible transition matrices that can appear, hence we will eventually produce a matrix $M_j = M_i \in \mathcal{M}$. Since this matrix is Perron-Frobenius, the dilatation $\lambda$ of $\psi$ is an eigenvalue with strictly positive eigenvectors $\mu_i$ and $\mu_j$. Moreover, $\lambda$ is simple, so in fact $\mu_j$ is a scalar multiple of $\mu_i$. According to Corollary \ref{cor:SplitPos}, each tight splitting reduces one of the entries of this eigenvector, so recurring to a matrix in $\mathcal{M}$ implies that every entry of $\mu$ has been reduced, i.e. that every real edge of $\tau_0$ has been split over. In particular, the stems of $\tau_0$ have been split over. The preceding paragraph now implies that the algorithm must terminate in finite time.

Repeating this algorithm sufficiently many times will eventually reduce $J$ to $0$, proving the theorem.
\end{proof}

\begin{proof}[proof of Theorem \ref{thm:221}]
Note that in the stratum $(2;1^5;3)$, there are only two classes of standardly-embedded train tracks without joints: those shown in Figure \ref{fig:TT5}. By Theorem \ref{carrythm2}, any pseudo-Anosov in this stratum is conjugate to one carried by either the Peacock or the Snail. We will argue that any pseudo-Anosov $\psi$ carried by the Snail tightly splits to one carried by the Peacock.

First, observe that $\psi$ must split at the unique valence-3 switch of the infinitesimal triangle in the Snail, by Corollary \ref{cor:maxsplit}. Either a left or right split at this vertex yields a pseudo-Anosov $\psi'$ conjugate to $\psi$, and carried by a track $\tau'$ with a unique two-valent vertex $v$ at a puncture. This vertex $v$ is again splittable by Corollary \ref{cor:maxsplit}. At $v$, note that $\psi'$ splits either to another map carried by $\tau'$, with strictly smaller edge weight on the edge running between two punctures, or to a map carried by the Peacock. In particular, after sufficiently many splits, $\psi'$ splits to a pseudo-Anosov carried by the Peacock.
\end{proof}

\begin{proof}[proof of Theorem \ref{carrythm}]
Note that if $\psi:S\to S$ has the given singularity type, we may cap-off $\psi$ to a pseudo-Anosov $\widehat{\psi}$ on the closed genus-two surface $\widehat{S}$ and extend the foliations preserved by $\psi$ along the capping disk. In this case, the 4-prong singularity $p$ in the capping disk is the unique 4-prong singularity of $\widehat{\psi}$. In particular, $\widehat{\psi}$ commutes with the hyperelliptic involution $\iota$ on $\widehat{S}$ and $p$ is fixed by $\iota$, as in e.g. Lemma 3.7 of \cite{BHS}. And, because $p$ is fixed by $\iota$, we see that $\psi$ commutes with the hyperelliptic involution on $S$, as well. We may then quotient $\psi$ to a pseudo-Anosov 5-braid $\beta$, and from here the techniques of Section \ref{sec:split} apply. Theorem \ref{thm:221} implies that $\beta$ is carried by the ``Peacock" train track depicted in Figure \ref{fig:TT5}, and we can then lift this track to $S$ as described in subsection \ref{liftsec}.
\end{proof}

\printbibliography
\end{document}